\documentclass [12pt,oneside]{amsart}
\usepackage{color}
\usepackage{amssymb,amsmath,amsthm,amsfonts}
\usepackage{enumerate}
\usepackage{empheq}
\usepackage[all]{xy}
\usepackage{verbatim}
\usepackage[normalem]{ulem}
\usepackage{todonotes}
\usepackage[bookmarks,colorlinks,breaklinks]{hyperref}  
\hypersetup{linkcolor=blue,citecolor=blue,filecolor=blue,urlcolor=blue} 

\numberwithin{equation}{section}
\numberwithin{figure}{section}

\newtheorem{theorem}{Theorem}[section]
\newtheorem{proposition}[theorem]{Proposition}

\newtheorem{lemma}[theorem]{Lemma}

\newtheorem{cor}[theorem]{Corollary}

\theoremstyle{definition}

\newtheorem{remark}[theorem]{Remark}

\definecolor{myblue}{rgb}{0.6, 0.9, 1}

\makeatletter

\newcommand{\Rmnum}[1]{\expandafter\@slowromancap\romannumeral #1@}
\makeatother

\definecolor{myblue}{rgb}{0.6, 0.9, 1}
\definecolor{mygreen}{rgb}{0,0,1}
\definecolor{purple}{rgb}{0.6,0.2,1}
\definecolor{orange}{rgb}{0.8,0,0.2}

\newcommand{\bC}{\mathbb{C}}
\newcommand{\bP}{\mathbb{P}}
\newcommand{\C}{\mathbb{C}}

\newcommand{\bZ}{\mathbb{Z}}
\newcommand{\bQ}{\mathbb{Q}}
\newcommand{\bR}{\mathbb{R}}
\newcommand{\R}{\mathbb{R}}
\newcommand{\bN}{\mathbb{N}}

\newcommand{\mcD}{\mathcal{D}}

\newcommand{\mcB}{\mathcal{B}}
\newcommand{\mcE}{\mathcal{E}}

\newcommand{\ord}{\operatorname{ord}}

\newcommand{\eps}{\varepsilon}

\newcommand{\la}{\lambda}

\newcommand{\pr}{\mathbb{P}}

\newcommand{\Tor}{\operatorname{Tor}}
\newcommand{\Gal}{\operatorname{Gal}}
\newcommand{\Kbar}{\overline{K}}
\newcommand{\Qbar}{\overline{\bQ}}
\newcommand{\kbar}{\overline{k}}

\newcommand{\supp}{\operatorname{supp}}

\newcommand{\<}{\langle}
\renewcommand{\>}{\rangle}
\newcommand\Lbar {\overline{L}}
\newcommand\Mbar {\overline{M}}
\newcommand\Dbar {\overline{D}}
\newcommand{\rank}{\operatorname{rank}}
\newcommand\volhat{\widehat{\mathrm{vol}}}
\newcommand\iso{\simeq}
\newcommand\essmin{\mathrm{ess.min.}}
\newcommand\Stilde{\widetilde{S}}
\newcommand\Hyp{\mathbb{H}}
\newcommand\an{\mathrm{an}}
\newcommand\vol{\operatorname{vol}}
\newcommand\deghat{\widehat{\deg}}

\begin{document}
	\title{
	Elliptic surfaces and intersections of adelic $\R$-divisors}
	
	\author{Laura De Marco and Niki Myrto Mavraki}
	\email{demarco@math.harvard.edu}
	\email{mavraki@math.harvard.edu}
	
	\date{\today}
	
\begin{abstract}
Suppose $\mcE \to B$ is a non-isotrivial elliptic surface defined over a number field, for smooth projective curve $B$.  Let $k$ denote the function field $\Qbar(B)$ and $E$ the associated elliptic curve over $k$.  In this article, we construct adelically metrized $\R$-divisors $\Dbar_X$ on the base curve $B$ over a number field, for each $X \in E(k)\otimes \R$.  We prove non-degeneracy of the Arakelov-Zhang intersection numbers $\Dbar_X\cdot\Dbar_Y$, as a biquadratic form on $E(k)\otimes \R$.  As a consequence, we have the following Bogomolov-type statement for the N\'eron-Tate height functions on the fibers $E_t(\Qbar)$ of $\mcE$ over $t \in B(\Qbar)$:   given points $P_1, \ldots, P_m \in E(k)$ with $m\geq 2$, there exist an infinite sequence $\{t_n\}\subset B(\Qbar)$  and small-height perturbations $P_{i,t_n}' \in E_{t_n}(\Qbar)$ of specializations $P_{i,t_n}$ so that the set $\{P_{1, t_n}', \ldots, P_{m,t_n}'\}$ satisfies at least {\em two} independent linear relations for all $n$, if and only if the points $P_1, \ldots, P_m$ are linearly dependent in $E(k)$.  This gives a new proof of results of Masser and Zannier \cite{Masser:Zannier, Masser:Zannier:2} and of Barroero and Capuano \cite{Barroero:Capuano} and extends our earlier results \cite{DM:variation}.  In the Appendix, we prove an equidistribution theorem for adelically metrized $\R$-divisors on projective varieties (over a number field) using results of Moriwaki \cite{Moriwaki:Memoir}, extending the equidistribution theorem of Yuan \cite{Yuan:equidistribution}.
\end{abstract}
	
\maketitle

\bigskip
\section{Introduction}
Suppose $\mcE \to B$ is an elliptic surface defined over a number field $K$.  That is, $\mcE$ is a projective surface, $B$ is a smooth projective curve, and there exists a section $O: B\to \mcE$, all defined over $K$, so that all but finitely many fibers $E_t$, for $t\in B(\overline{K})$, are smooth elliptic curves with zero $O_t$. We say that the elliptic surface $\mcE \to B$ is isotrivial if all of the smooth fibers $E_t$ are isomorphic over $\Kbar$.  Let $k$ denote the function field $\Kbar(B)$; we also view the surface $\mcE$ as an elliptic curve $E$ over the field $k$.  
	
In this article, we study the geometry and arithmetic of the set $E(k)$ of rational points over the function field $k$ when $\mcE \to B$ is not isotrivial.  To this end, we consider height functions associated to adelically metrized $\R$-divisors on the base curve $B$ over the number field $K$.   We study the Arakelov-Zhang intersection of these metrized $\R$-divisors and prove that it induces a non-degenerate biquadratic form on $E(k) \otimes \R$.  We relate this theorem to existing results, and provide, for example, a new proof of results of Masser and Zannier and of Barroero and Capuano on linear relations between specializations of independent sections.

\subsection{Heights and the Arakelov-Zhang intersection of points in $E(k)$}  Assume that $\mcE\to B$ is not isotrivial.  Let $\hat{h}_E$ denote the  N\'eron-Tate canonical height on $E(\kbar)$, associated to the choice of divisor $O$ on $E$; let $\hat{h}_{E_t}$ denote the corresponding canonical height on the smooth fibers $E_t(\Kbar)$ for (all but finitely many) $t\in B(\Kbar)$.  By non-isotriviality, a point $P \in E(k)$ satisfies $\hat{h}_E(P) =0$ if and only if it is torsion on $E$.  We denote the specializations of $P$ by $P_t$ in the fiber $E_t$.  Tate showed in \cite{Tate:variation} that the canonical height function 
\begin{equation} \label{height function}
	h_P(t) := \hat{h}_{E_t}(P_t)
\end{equation}
is a Weil height on the base curve $B(\Kbar)$, up to a bounded error.  More precisely, there exists a $\bQ$-divisor $D_P$ on $B$ of degree equal to $\hat{h}_E(P)$ so that $h_P(t) = h_{D_P}(t) + O(1)$, where $h_{D_P}$ is a Weil height on $B(\Kbar)$ associated to $D_P$. In \cite{DM:variation}, we showed that we can also understand the small values of the function \eqref{height function}, from the point of view of equidistribution.  Assume that $\hat{h}_E(P)>0$ (so that the function $h_P$ is nontrivial) and that, as a section, $P: B\to \mcE$ is defined over the number field $K$.   Building on work of Silverman \cite{Silverman:VCHI, Silverman:VCHII, Silverman:VCHIII}, we showed that $h_P$ is the height induced by an ample line bundle on $B$ (with divisor $D_P$) equipped with a continuous, adelic metric of non-negative curvature defined over $K$, denoted by $\Dbar_P$ and satisfying 
	$$\Dbar_P \cdot \Dbar_P = 0$$
for the Arakelov-Zhang intersection number introduced in \cite{Zhang:adelic}.  In particular, we can then apply the equidistribution theorems of \cite{ChambertLoir:equidistribution, Thuillier:these, Yuan:equidistribution} to deduce that the $\Gal(\Kbar/K)$-orbits of points $t_n\in B(\Kbar)$ with height $h_P(t_n) \to 0$ are uniformly distributed on $B(\bC)$ with respect to the curvature distribution $\omega_P$ for $\Dbar_P$ at an archimedean place of $K$.  A similar equidistribution occurs at each place $v$ of $K$ to a measure $\omega_{P,v}$ on the Berkovich analytification $B_v^{\an}$ \cite[Corollary 1.2]{DM:variation}. 
	
As a consequence of our main result in \cite{DM:variation}, and combined with the results of Masser and Zannier \cite{Masser:Zannier, Masser:Zannier:2}, we have 
\begin{eqnarray} \label{nondegenerate}
	\Dbar_P \cdot \Dbar_Q \geq 0 &\mbox{ for all } &  P, Q \in E(k), \mbox{ and } \\
	\Dbar_P \cdot \Dbar_Q = 0 
	&\iff& \mbox{either } P \mbox{ or } Q \mbox{ is torsion, or }  \nonumber \\
	&&  \exists \; \alpha >0 \mbox{ such that } h_P(t) = \alpha \, h_Q(t) \mbox{ for all } t \in B(\Kbar) \nonumber \\
	&\iff&  \exists \; (n,m) \in \bZ^2 \setminus\{(0,0)\} \mbox{ such that } n P = m Q  \nonumber
\end{eqnarray}
In particular, as the N\'eron-Tate bilinear form $\<P,Q\>_E := \frac12 \left( \hat{h}_E(P+Q) - \; \hat{h}_E(P) - \hat{h}_E(Q)\right)$ is positive definite on $E(k)\otimes \R$, we have 
\begin{equation}\label{NT 0}
	\Dbar_P \cdot \Dbar_Q = 0 \quad \iff \quad \hat{h}_E(P) \hat{h}_E(Q) = \<P,Q\>_E^2
\end{equation}
for all $P, Q\in E(k)$.  
	
The main result of this article is the proof of a stronger version of \eqref{NT 0}:
	
\begin{theorem}  \label{Lambda}  
Let $\mcE \to B$ be a non-isotrivial elliptic surface defined over a number field $K$.  Let $E$ be the corresponding elliptic curve over the field $k = \Kbar(B)$.  There exists a constant $c >0$ so that 
	$$c \left( \hat{h}_E(P) \hat{h}_E(Q) - \<P,Q\>_E^2\right) \; \leq \; \Dbar_P \cdot \Dbar_Q \; \leq \; c^{-1} \left( \hat{h}_E(P) \hat{h}_E(Q) - \<P,Q\>_E^2\right)$$
for all $P, Q \in E(k)$, where $\< \cdot, \cdot \>_E$ is the N\'eron-Tate bilinear form on $E(k)$.   
\end{theorem}
	
\noindent
The upper bound on $\Dbar_P \cdot \Dbar_Q$ in Theorem \ref{Lambda} is relatively straightforward.  The difficulty lies in the lower bound; in Section \ref{equivalences}, we observe that this is equivalent to proving that $\Dbar_X \cdot \Dbar_Y > 0$ for all independent $X, Y \in 
E(k)\otimes \R$.

\subsection{Motivation and context}
Theorem \ref{Lambda} was inspired by the statements and proofs of the Bogomolov Conjecture \cite{Ullmo:Bogomolov, Zhang:Bogomolov, Szpiro:Ullmo:Zhang}, extending Raynaud's theorem that settled the Manin-Mumford Conjecture \cite{Raynaud:1}, and the ``Mordell-Lang plus Bogomolov" results of Poonen \cite{Poonen:bogplus} and Zhang \cite{Zhang:bogplus}, in the spirit of the conjectures of Pink \cite{Pink} and Zilber \cite{Zilber}.   Moreover, as we will explain in Section \ref{equivalences}, we view Theorem \ref{Lambda} as an analog of Zhang's Conjecture \cite[\S4]{Zhang:ICM}; the conjecture was formulated for families of abelian varieties $\mathcal{A} \to B$ of relative dimension $>1$ and does not hold as stated for elliptic surfaces \cite[\S4 Remark 3]{Zhang:ICM}.  (See \cite{Zannier:book} for background and additional references.)
	
Theorem \ref{Lambda} can be seen as a Bogomolov-type bound.  The intersection number $\Dbar_P \cdot \Dbar_Q$ is related to the small values of the heights $\hat{h}_{E_t}(P_t) + \hat{h}_{E_t}(Q_t)$ in the fibers $E_t(\Kbar)$.  Indeed, as a consequence of Zhang's Inequality \cite[Theorem 1.10]{Zhang:adelic} applied to the sum $\Dbar_P + \Dbar_Q$, and the fact that $h_P(t) \geq 0$ at all points $t\in B(\Kbar)$ for every $P \in E(k)$ \cite[Proposition 4.3]{DM:variation}, we have 
\begin{equation} \label{Zhang for the sum}
	\frac12 \; \essmin \left( h_P + h_Q\right) \; \leq \;  \frac{ \Dbar_P \cdot \Dbar_Q }{\hat{h}_E(P)+\hat{h}_E(Q)} \; \leq  \;  \essmin \left( h_P + h_Q\right)
\end{equation} 
for every pair of non-torsion $P, Q \in E(k)$.  The essential minimum is defined by $\essmin(f) = \sup_F \inf_{x \in B\setminus F} f(x)$ over all finite sets $F$ in $B(\Kbar)$. Bogomolov-type bounds have found many applications in problems of unlikely intersections.  In Section \ref{equivalences}, we explain that Theorem \ref{Lambda} is equivalent to the following:
	
\begin{theorem}  \label{scheme} 
Let $\mcE\to B$ be a non-isotrivial elliptic surface defined over a number field, and let $\pi: \mathcal{E}^m \to B$ be its $m$-th fibered power with $m\geq 2$.  Let $\mathcal{E}^{m,\{2\}}$ denote the union of flat subgroup schemes of $\mathcal{E}^m$ of codimension at least 2, and consider the tubular neighborhood
	$$T(\mcE^{m,\{2\}},\epsilon)=\left\{P\in\mathcal{E}^m(\Qbar) ~:~ \exists P'\in \mathcal{E}^{m,\{2\}}(\overline{\mathbb{Q}}) \mbox{ with }\pi(P) = \pi(P') \mbox{ and } \hat{h}_{\mcE^m_{\pi(P)}}(P-P') \leq \epsilon\right\}.$$
Then, for any irreducible curve $C$ in $\mathcal{E}^m$, defined over a number field and dominating $B$, there exists $\epsilon>0$ such that 
	$$C\; \cap \; T(\mathcal{E}^{m,\{2\}},\epsilon)$$ 
is contained in a finite union of flat subgroup schemes of positive codimension in $\mcE^m$.
\end{theorem}
	
\noindent
See, e.g., \cite[Lemma 2.2]{Barroero:Capuano} for definitions and a classification of flat subgroup schemes.  Our main result in \cite{DM:variation} treated the intersections of $C$ with the smaller tube $T(\mathcal{E}^{m,\{m\}},\epsilon)$, the torsion subgroups.

The conclusion of Theorem \ref{scheme} with $\epsilon=0$ is a result of Barroero and Capuano \cite[Theorem 2.1]{Barroero:Capuano}:  using techniques involving o-minimality and transcendence theory, similar to those of \cite{Masser:Zannier, Masser:Zannier:2} (which treated the intersections of curves $C$ with  $T(\mathcal{E}^{m,\{m\}},0)$), they show that  $C\,\cap \, T(\mathcal{E}^{m, \{2\}},0)$ is contained in a finite union of flat subgroup schemes of positive codimension.  Thus Theorem \ref{scheme} may be seen as a Bogomolov-type extension of \cite[Theorem 2.1]{Barroero:Capuano}, while providing a new proof of results in \cite{Barroero:Capuano, Masser:Zannier, Masser:Zannier:2}.  The result in \cite{Barroero:Capuano} is extended in \cite{Barroero1,Barroero2} who proved Pink's conjecture \cite[Conjecture 6.1]{Pink} for curves in $\mathcal{E}^m$.  We may also view Theorem \ref{scheme} as a Bogolomolov-type extension of a special case of Pink's conjecture \cite[Conjecture 6.1]{Pink}.  However, Pink's conjecture also considers algebraic subgroups of codimension at least $2$ within fibers having complex multiplication, which we do not treat here, having restricted our study to flat subgroup schemes.

If the elliptic surface $\mathcal{E}\to B$ is isotrivial, the conclusion of Theorem \ref{scheme} with $\epsilon =0$ was established by Viada \cite{ViadaANT} and Galateau \cite{Galateau}.  Moreover, in this isotrivial setting, Viada proved the analogue of Theorem \ref{scheme} (for positive effective $\epsilon$) in \cite[Theorem 1.4]{ViadaIMRN}, \cite[Theorem 1.2]{ViadaANT}, providing in particular new proofs of instances of earlier results by Poonen \cite{Poonen:bogplus} and Zhang \cite{Zhang:bogplus} and extending the work in \cite{Remond:Viada}.  It is worth pointing out that the aforementioned results invoked a different Bogomolov-type bound than the one in Theorem \ref{Lambda}, established by Galateau \cite{Galateau}.  In the case $\epsilon=0$, an analogous statement for curves in constant abelian varieties is established in \cite{HabeggerPila}. The authors use, amongst others, techniques from o-minimality.  In the setting of the multiplicative group $\mathbb{G}^n_m$, Habegger established results of this flavor in arbitrary dimension \cite{Habegger:bogomolov}, generalizing a result of Bombieri-Masser-Zannier \cite{BMZ:1999}.  
	
We remark that the analogues of Theorems \ref{Lambda} and \ref{scheme} can be formulated for arbitrary fiber products of elliptic surfaces over a given base curve $B$, as we did in \cite[Theorem 1.4]{DM:variation}.  For example, Theorem \ref{Lambda} would assert that $\overline{D}_{\mathcal{E},P}\cdot \overline{D}_{\mathcal{F},Q}$ is comparable with $\hat{h}_{E}(P)\hat{h}_{F}(Q)$ if the two non-isotrivial elliptic surfaces $\mathcal{E}\to B$ and $\mathcal{F}\to B$ are not isogenous.  Theorem \ref{scheme} would read exactly the same upon replacing $\mathcal{E}^m$ in the statement with the fibered product $\mathcal{E}_1\times_B\cdots\times_B\mathcal{E}_m$ of any $m\ge 2$ non-isotrivial elliptic surfaces $\mathcal{E}_i\to B$.  Our methods here would yield these results and, in particular, \cite[Theorem 1.1]{BC2} of Barroero and Capuano. We omit them in this article to simplify our exposition.

\subsection{Metrized $\R$-divisors on curves and proof strategy}
For each $t \in B(\Kbar)$ with $E_t$ smooth, the canonical height $\hat{h}_{E_t}$ induces a positive definite quadratic form on $E_t(\Kbar) \otimes \bR$; see e.g. \cite[Ch. VIII, Prop. 9.6]{Silverman:Elliptic}.  The height functions $h_P$ on $B(\Kbar)$, defined by \eqref{height function} for $P \in E(k)$, therefore make sense for elements of the finite-dimensional vector space $E(k)\otimes \bR$.  In Theorem \ref{real point divisor theorem}, we prove that every nonzero element $X \in E(k) \otimes \R$ gives rise to a continuous, adelic, semi-positive metrization $\Dbar_X$ of an ample $\R$-divisor on the base curve $B$, defined over a number field $K$, with height function $h_X(t) = \hat{h}_{E_t}(X_t)$ for $t\in B(\Kbar)$ when $E_t$ is smooth, satisfying $\Dbar_X \cdot \Dbar_X = 0$. 
	
Consequently, we are able to employ results of Moriwaki \cite{Moriwaki:Memoir} in our proof of Theorems \ref{Lambda} and \ref{scheme}.  Specifically, we use his arithmetic Hodge index theorem for adelically metrized $\R$-divisors on curves defined over a number field \cite[Corollary 7.1.2]{Moriwaki:Memoir} to deduce that $\Dbar_X \cdot \Dbar_Y = 0$ for $X, Y \in E(k) \otimes\R$ implies that $\Dbar_X \iso \Dbar_Y$.  As we will show in Section \ref{equivalences}, the proofs of Theorems \ref{Lambda} and \ref{scheme} are then reduced to showing which points $X, Y$ give rise to isomorphic metrized $\R$-divisors on $B$.  
	
To complete the proofs of Theorems \ref{Lambda} and \ref{scheme}, we examine the curvature distributions for $\Dbar_X$.  Fix an embedding of the number field $K$ into $\C$.  In \cite{DM:variation}, the curvature measure $\omega_P$ for the metrized divisor $\Dbar_P$ of $P \in E(k)$, at the given archimedean place, is computed as the pullback by $P$ of a certain $(1,1)$-form on $\mcE(\C)$, via a dynamical construction.  In \cite{CDMZ}, it is shown that $\omega_P = db_1\wedge db_2$ in the Betti coordinates $(b_1, b_2)$ of $P$.   We explain in Section \ref{measure section} that elements $X \in E(k) \otimes \R$ are also represented by holomorphic curves in the surface $\mcE$, and the Betti coordinates of $X$ are real linear combinations of the Betti coordinates of points $P_i \in E(k)$.   We use this to prove that the measure $\omega_X$, at a single archimedean place of the number field $K$, is enough to uniquely determine the pair of points $X$ and $-X$:  
	
\begin{theorem} \label{measures}
Fix $X$ and $Y$ in $E(k) \otimes \R$ and an archimedean place of the number field $K$.  Let $\omega_X$ and $\omega_Y$ denote the curvature distributions on $B(\C)$ at this place for the adelically metrized $\R$-divisors $\Dbar_X$ and $\Dbar_Y$.  Then 
	$$\omega_X = \omega_Y \iff X = \pm Y .$$  
\end{theorem}
	
We are grateful to Lars K\"uhne for helping us with the proof of Theorem \ref{measures}; we use the holomorphic-antiholomorphic trick of Andr\'e, Corvaja, and Zannier \cite[\S5]{ACZ:Betti} and a transcendence result of Bertrand \cite[Th\'eor\`eme 5]{Bertrand}.   A special case of Theorem \ref{measures} was proved by a different method in \cite[Proposition 1.9]{DWY:Lattes}.

\subsection{Small points}
In the Appendix, we show that heights associated to semipositive metrized $\R$-divisors satisfy an equidistribution law.  As we shall see, Corollary \ref{equidistribution} applies to sequences in the base curve $B$ where the specializations of points in $E(k)$ satisfy non-trivial linear relations.  For example, generalizing \cite[Corollary 1.2]{DM:variation},  we obtain:
	
\begin{theorem} \label{geometry of relations}
Let $\mcE \to B$ be a non-isotrivial elliptic surface defined over a number field $K$, and let $E$ be the corresponding elliptic curve over the field $k = \Kbar(B)$.   Suppose that $P_1, \ldots, P_m$ is a collection of $m\geq 1$ linearly independent points in $E(k)$, also defined over $K$ as sections of $\mcE\to B$.  Suppose that $\{t_n\}\subset B(\Kbar)$ is a non-repeating sequence where
\begin{equation} \label{relationX}
		a_{1,n} P_{1,t_n} + a_{2,n} P_{2,t_n} + \cdots + a_{m,n} P_{m,t_n} = O_{t_n}
\end{equation}
for $a_{i,n} \in \bZ$, with $[a_{1,n}:\cdots:a_{m,n}]\to [x_1:\cdots:x_m]$ in $\mathbb{RP}^{m-1}$ as $n\to\infty$.  Set 
	$$X = x_1P_1 + \cdots + x_n P_n \;  \in \;  E(k)\otimes \R.$$
Then 
	$$h_X(t_n) \to 0$$
for the height function associated to the metrized $\R$-divisor $\Dbar_X$.  Moreover, for each place $v$ of $K$, the $\Gal(\Kbar/K)$-orbits of $t_n$ in $B(\Kbar)$ are uniformly distributed on $B_v^{\an}$ with respect to the probability measure
$$\mu_{X,v} \; := \; \frac{1}{\hat{h}_E(X)} \, \omega_{X,v} \; =  \;  \frac{1}{\hat{h}_E(X)}\left(\sum_i \left(x_i^2 - \sum_{j\not=i} x_ix_j \right)  \omega_{P_i,v} + \sum_{i < j} x_ix_j \; \omega_{P_i + P_j, v}\right).  $$
\end{theorem}
	
\noindent
A sequence $\{t_n\}_{n\geq 0}$ is said to be non-repeating if $t_n \not= t_m$ for all $n\not= m$.

\begin{remark}
For nonzero $X \in E(k)\otimes \R$, the height $h_X$ will have only finitely many zeros unless a positive real multiple $cX$ is represented by an element of $E(k)$; see Proposition \ref{finite zeros}.  On the other hand, there is always an infinite sequence $\{t_n\}$ for which \eqref{relationX} is satisfied so that $\essmin (h_X) = 0$; see Proposition \ref{small points exist}. 
\end{remark}

\subsection{Example}
Let $E_t$ be the Legendre elliptic curve defined by 
	$$y^2 = x(x-1)(x-t)$$
for $t \in \Qbar\setminus\{0,1\}$.  By filling in the family over $t = 0,1,\infty$, we obtain an elliptic surface $\mcE \to B$ with $B = \bP^1$ defined over $\bQ$.  Here $k = \Qbar(t)$.  It is easy to see that $\rank E(k) = 0$.  However, by choosing any collection of $m$ distinct points $x_1, x_2, \ldots, x_m \in \bP^1(\Qbar)\setminus\{0,1,\infty\}$, we can construct an elliptic surface $\mcE'\to B'$ with $\rank E'(k') \geq m$ where $k' = \Qbar(B')$.  Indeed, we let $P_{x_i}$ be a point with constant $x$-coordinate equal to $x_i$.  As the points $x_i$ are distinct, the structure of the field extensions $k_i/k$, determined by each $P_{x_i}$, implies that the points must be independent.  We pass to a branched cover $B' \to B$ so that each $P_{x_i}$ defines a section over $B'$ and set $k' = \Qbar(B')$.  These examples were considered in \cite{Masser:Zannier} and the associated measures $\omega_{P_{x_i}}$ on $B'(\C)$ (or rather, their projections to $B= \bP^1$) were computed in \cite{DWY:Lattes}.

\subsection{Outline of the article}  In Section \ref{R-divisors}, we fix some notation and introduce metrizations on $\R$-divisors on curves defined over a number field, and we examine their intersection numbers. In Section \ref{R points}, we prove that each nonzero element $X \in E(k) \otimes \R$ induces a continuous, adelic, semipositive metrization $\Dbar_X$ on an ample $\R$-divisor on the base curve $B$.  In Section \ref{small section}, we study the sequences of small points for the height function $h_X$ on $B(\Qbar)$ associated to $\Dbar_X$.  In Section \ref{biquadratic section} we lay out the basic properties of the intersection number $(X,Y) \mapsto \Dbar_X \cdot \Dbar_Y$ as a biquadratic form on the vector space $E(k) \otimes \R$.  In Section \ref{equivalences}, we analyze the significance of $\Dbar_X \cdot \Dbar_Y = 0$ for nonzero $X, Y \in E(k) \otimes \R$, and we explain how to relate Theorems \ref{Lambda} and \ref{scheme}.  We provide a list of equivalent formulations of these theorems in Theorem \ref{TFAE}, including one inspired by Zhang's Conjecture in \cite{Zhang:ICM}.  Section \ref{measure section} contains a proof of Theorem \ref{measures}, and we complete the proofs of Theorems \ref{Lambda} and \ref{scheme} in Section \ref{final proofs}.  In the Appendix, we provide a proof of equidistribution results for heights associated to $\R$-divisors on projective varieties.

\subsection{Acknowledgements}  We spoke with many people about this work, and we are grateful to all of them for helpful discussions, including Fabrizio Barroero, Daniel Bertrand, Laura Capuano, Gabriel Dill, Philipp Habegger, Harry Schmidt, Xinyi Yuan, and Umberto Zannier.  We are especially thankful for the assistance from Lars K\"uhne in our proof of Theorem \ref{measures}.  We'd also like to express our profound gratitude to the anonymous referees for their very useful feedback and suggestions, encouraging us to simplify our presentation and generalize the result in the Appendix to arbitrary dimension.

\bigskip
\section{$\R$-divisors on curves and arithmetic intersection}
\label{R-divisors}
	
In this section, we introduce metrizations on $\R$-divisors on curves, following Moriwaki \cite{Moriwaki:Memoir}, and their intersection numbers.

\subsection{Notation}
Here, and throughout this article, $K$ denotes a number field.  We let $M_K$ denote its set of places, with absolute values $|\cdot|_v$ satisfying the product formula:
\begin{equation} \label{product formula}
	\prod_{v\in M_K} |x|_v^{[K_v: \bQ_v]} = 1
\end{equation}
for all nonzero $x$ in $K$.   Here $K_v$ denotes the completion of $K$ with respect to $|\cdot|_v$.  We set 
\begin{equation} \label{rv}
	r_v := \frac{[K_v:\bQ_v]}{[K:\bQ]}.
\end{equation}
For each place $v \in M_K$, we let $\C_v$ denote the completion of an algebraic closure of $K_v$.
	
We let $B$ denote a smooth projective curve defined over a number field $K$.  For each $v\in M_K$, we let $B_v^{\an}$ denote the Berkovich analytification of $B$ over the field $\C_v$.
	
We let $\mathrm{Div}_{\bZ}(B)$ denote the group of divisors on $B$.  
	
Throughout, $k$ denotes the function field $\Kbar(B)$.  Its places are in one-to-one correspondence with the elements $t\in B(\Kbar)$, with absolute values given by $|f|_t = \exp(-\ord_t(f))$ for each nonzero $f\in \Kbar(B)$.  

\subsection{Metrizations of $\R$-divisors on curves}  \label{metrizations}
Let $B$ be a smooth projective curve defined over a number field $K$.  Let $D = \sum_i a_i D_i$ be an ample $\R$-divisor on $B$, with $a_i \in \R$ and $D_i \in \mathrm{Div}_{\bZ}(B)$ with support in $B(\Kbar)$, invariant under the action of $\Gal(\Kbar/K)$.  By rewriting the sum if necessary, we may assume that each $D_i$ is associated to an ample line bundle $L_i$ that extends over the Berkovich analytification $B^{\an}_v$ for each place $v$ of $K$. 
	
A {\bf continuous, adelic metrization} for $D$ is a collection of continuous functions 
	$$g_v : B^{\an}_v \setminus \supp D \to \R$$
for $v \in M_K$, such that 
\begin{enumerate}
	\item For each $v$, the locally-defined function $\psi_v := g_v + \sum_i a_i \log|f_i|_v$ extends continuously to the support of $D$, where $f_i$ is a local defining equation for $D_i$ defined over $K$; 
	\item there exists a model $(\mcB, \mcD)$ of $(B,D)$ over the ring of integers $O_K$  so that $g_v$ is the associated model function for all but finitely many $v$, or equivalently, the function $\psi_v \equiv 0$ at all but finitely many places $v$ for the associated $\{f_i\}$ near each element of $\supp D$.
\end{enumerate}
See \cite[\S0.2]{Moriwaki:Memoir} and \cite[\S1.3.2]{ChambertLoir:survey} for the definition of model functions.  We denote this data by $\Dbar = (D, \{g_v\}_{v \in M_K})$.

The metrization is {\bf semipositive} if each $g_v$ is subharmonic on $B^{\an}_v \setminus \supp D$.    An $\R$-divisor $D$ on $B$ and collection of continuous functions $g_v : B_v^{\an}\setminus\supp D \to \R$, for $v \in M_K$, is said to be {\bf integrable} if $D = D_1 - D_2$ and $g_v = g_{v,1} - g_{v,2}$ for two adelic, semipositive metrizations on ample $\R$-divisors $\Dbar_i = (D_i, \{g_{i,v}\})$.  We write $\Dbar = \Dbar_1 - \Dbar_2$.  An associated height function is given by $h_{\Dbar} = h_{\Dbar_1}- h_{\Dbar_2}$.

Moriwaki calls a semipositive $\Dbar$ a {\em relatively nef} adelic arithmetic $\R$-divisor on $B$ \cite{Moriwaki:Memoir}.  This extends Zhang's notion of an adelic, semipositive metric on a line bundle to $\R$-divisors \cite{Zhang:adelic}. Indeed, for $D$ an ample divisor on $B$ associated to a line bundle $L$, equipped with an adelic metric $\{\|\cdot \|_v\}_{v\in M_K}$, and $s$ a meromorphic section of $L$ with $(s) = D$, we put $g_v = -\log \|s\|_v$ at each place $v$ of the number field $K$. 
	
For any integrable $\Dbar$, we let $\omega_{\Dbar, v}$ denote its {\bf curvature distribution} on $B_v^{\an}$; by definition, this is a (signed) measure of total mass $\deg D$, equal to the Laplacian of $g_v$ away from $\supp D$.  See, for example, \cite{BRbook} for more information about the distribution-valued Laplacian on Berkovich curves.  For semipositive $\Dbar$, the measure $\omega_{\Dbar_v}$ is positive, and its associated probability measure is denoted by 
	$$\mu_{\Dbar,v} := \frac{1}{\deg D}  \; \omega_{\Dbar,v}.$$

There is an associated {\bf height function} on $B(\Kbar)$ defined by
\begin{equation} \label{R height}
	h_{\Dbar}(x) :=  \sum_{v\in M_K} \frac{r_v}{|\Gal(\Kbar/K) \cdot x|} \sum_{x' \in \Gal(\Kbar/K)\cdot x} g_v(x'),
\end{equation}
for $x \not\in \supp D$.  Recall that $r_v$ was defined in \eqref{rv}.  For any rational function $\phi$ on $B$ defined over $K$, and for any real $a\in \R$, note that 
	$$h_{\Dbar}(x) =  \sum_{v \in M_K} \frac{r_v}{|\Gal(\Kbar/K) \cdot x|} \sum_{x' \in \Gal(\Kbar/K)\cdot x} \left(g_v - a \log|\phi|_v\right)(x')$$
away from $(\supp D)  \cup  (\supp (\phi))$, from the product formula \eqref{product formula}.  This allows definition \eqref{R height} to extend to the points $x \in \supp D$, by choosing any $\phi$ so that $x \in \supp (\phi)$ and $a$ so that $g_v - a \log|\phi|_v$ extends continuously at $x$ for every $v$.  For an $\R$-divisor $D' = \sum_i b_i \, [x_i]$ with support in $B(\Kbar)$, we will write 
	$$h_{\Dbar}(D') := \sum_i b_i \, h_{\Dbar}(x_i).$$
	
\subsection{Intersection} 
For divisors $D_1, D_2\in \mathrm{Div}_{\bZ}(B)$ associated to line bundles $L_1$ and $L_2$, respectively, equipped with continuous, adelic metrics $\Dbar_1$ and $\Dbar_2$, the arithmetic intersection number is defined in \cite{Zhang:adelic} (see also \cite{ChambertLoir:survey}) as 
\begin{eqnarray} \label{line bundle pairing}
		\quad \Dbar_1 \cdot \Dbar_2 &:=& h_{\Dbar_1}((s_2)) +  \sum_{v \in M_K} r_v \int_{B_v^{\an}} (-\log\| s_2 \|_{\Dbar_2,v}) \, d\omega_{\Dbar_1,v} \\
		&=& h_{\Dbar_1}((s_2)) + h_{\Dbar_2}((s_1)) + \sum_{v \in M_K} r_v \int_{B_v^{\an}} (-\log\| s_2\|_{\Dbar_2,v}) \, (d\omega_{\Dbar_1,v}-\delta_{(s_1)}) \nonumber \\ 
		&=& h_{\Dbar_1}((s_2)) + h_{\Dbar_2}((s_1)) + 
		\sum_{v \in M_K} r_v \int_{B_v^{\an}} (-\log\| s_2 \|_{\Dbar_2,v}) \, \Delta( -\log\| s_1 \|_{\Dbar_1,v})  \nonumber \\	
		&=& \Dbar_2 \cdot \Dbar_1,   \nonumber
\end{eqnarray}
where $s_i$ is a meromorphic section of $L_i$ defined over $K$, for $i = 1, 2$, with divisors $(s_1)$ and $(s_2)$ of disjoint support.  For the continuous, adelic metrizations of $\R$-divisors, we extend by $\R$-linearity, so that 
\begin{equation} \label{pairing}
		\Dbar_1 \cdot \Dbar_2 = h_{\Dbar_1}(D_2) + \sum_{v \in M_K} r_v \int_{B_v^{\an}} g_{\Dbar_2,v} \, d\omega_{\Dbar_1,v} = \Dbar_2 \cdot \Dbar_1.
\end{equation}
	
\begin{remark}\label{intersectionvsdeg}
The intersection number \eqref{pairing} coincides with $\widehat\deg(\Dbar_1\, \Dbar_2)$ of \cite{Moriwaki:Memoir}.  Indeed, \cite[Theorem 4.1.3]{Moriwaki:Memoir} states that each $\Dbar$ can be uniformly approximated by metrizations associated to models, and it is known that the intersection numbers coincide for these model metrics \cite[Proposition 2.1.1]{Moriwaki:Dirichlet}.  
\end{remark}
	
Now suppose that $D$ is an ample $\R$-divisor on $B$.  We say $\Dbar$ is {\bf normalized} if its self-intersection number satisfies
	$$\Dbar \cdot \Dbar = 0.$$
Note that any continuous, adelic metrization on the ample $D$ can be normalized by adding a constant to $g_v$ at some place. 
	
For each $a\in \R$ and $\Dbar = (D, \{g_v\})$, we write $a \Dbar$ for the pair $(aD, \{ag_v\})$.  Normalized metrized divisors $\Dbar_1$ and $\Dbar_2$ on $B$ are {\bf isomorphic} (written as $\Dbar_1 \iso \Dbar_2$) if $\Dbar_1 - \Dbar_2$ is principal, meaning that there are rational functions $\phi_1, \ldots, \phi_m \in K(B)$ and real numbers $a_1, \ldots, a_m$, so that 
	$$\Dbar_1 - \Dbar_2 = \sum_{i=1}^m a_i \, \big( (\phi_i), \{-\log|\phi_i|_v\}_{v\in M_K}\big).$$
Note that by the product formula the height function $h_{\Dbar}$ depends only on the isomorphism class of $\Dbar$.    
	
We will make use of Moriwaki's arithmetic Hodge-index theorem in the following form:
	
\begin{theorem} \cite[Corollary 7.1.2]{Moriwaki:Memoir} \label{isomorphic} Suppose $\Dbar_1$ and $\Dbar_2$ are normalized continuous semipositive adelic metrizations on ample $\R$-divisors with $\deg D_1 = \deg D_2$.  Then $\Dbar_1\cdot \Dbar_2 \geq 0$, and $\Dbar_1\cdot \Dbar_2 = 0$ if and only if $\Dbar_1$ and $\Dbar_2$ are isomorphic.
\end{theorem}
	
\begin{proof}
Set $\Dbar = \Dbar_1 - \Dbar_2$, so that the underlying divisor $D$ has degree 0, and 
	$$\Dbar \cdot \Dbar = -2\; \Dbar_1 \cdot \Dbar_2$$
From \cite[Corollary 7.1.2]{Moriwaki:Memoir}, we have that $\Dbar \cdot \Dbar \leq 0$ with equality if and only if $\Dbar$ is principal, up to the addition of a constant $c \in \R$ to the metrization $g_v$ at some place $v$.  But then $\Dbar_1\cdot \Dbar_1 = \Dbar_2 \cdot \Dbar_2 + 2 \,c\,  r_v \deg D_2$ for this constant $c$, so the normalization of $\Dbar_1$ and $\Dbar_2$ implies that $c=0$.
\end{proof}

\subsection{Essential minima}  \label{minima}
Following \cite{Zhang:adelic}, the essential minimum of the height $h_{\Dbar}$ is defined as 
\begin{equation} \label{ess min}
	e_1(\Dbar) := \sup_{F} \inf_{x \in B(\Kbar)\setminus F} h_{\Dbar}(x),
\end{equation}
over all finite subsets $F$ of $B(\Kbar)$, and we put 
	$$e_2(\Dbar) := \inf_{x\in B(\Kbar)} h_{\Dbar}(x).$$
	
\begin{theorem} \cite[Theorem 1.10]{Zhang:adelic} \label{Zhang inequality}
For any adelic, semipositive metrization $\Dbar$ of an ample $\R$-divisor $D$, we have
	$$e_1(\Dbar) \geq \; \frac{\Dbar \cdot \Dbar}{2 \deg D} \;  \geq \frac12 \left( e_1(\Dbar) + e_2(\Dbar) \right)$$
\end{theorem}
	
\begin{proof}
Zhang proved the result for ample line bundles equipped with adelic, semipositive metrics \cite[Theorem 1.10]{Zhang:adelic}.  It holds also for metrizations of $\R$-divisors because the height function associated to an $\R$-divisor is a uniform limit of heights associated to $\bQ$-divisors, and the intersection number is a bilinear form on metrized divisors.
\end{proof}
	
Using the upper bound on $\Dbar\cdot \Dbar$ in Theorem \ref{Zhang inequality}, we can extend Theorem \ref{isomorphic} to:
	
\begin{theorem}  \label{isomorphic+} Suppose $\Dbar_1$ and $\Dbar_2$ are normalized semipositive adelic metrizations on ample $\R$-divisors of the same degree, and suppose the essential minimum of at least one of the $\Dbar_i$ is 0.  Then the following are equivalent: 
\begin{enumerate}
	\item $\Dbar_1\cdot \Dbar_2 = 0$ 
	\item  $\Dbar_1$ and $\Dbar_2$ are isomorphic
	\item  $h_{\Dbar_1} = h_{\Dbar_2}$ on $B(\Kbar)$
	\item $h_{\Dbar_1} = h_{\Dbar_2}$ at all but finitely many points of $B(\Kbar)$
	\item there exists an infinite non-repeating sequence $t_n$ in $B(\Kbar)$ for which 
$$\lim_{n\to\infty}  \left(h_{\Dbar_1}(t_n) + h_{\Dbar_2}(t_n)\right) = 0.$$
	\end{enumerate}
\end{theorem}
	
\begin{proof}
We have $(1) \iff (2)$ from Theorem \ref{isomorphic}.  The definition of the height function, in view of the product formula, implies that $(2) \implies (3)$, and we clearly have $(3) \implies (4)$.  The essential minimum being 0 for $\Dbar_1$ or for $\Dbar_2$ gives $(4)\implies (5)$.  Finally, assume $(5)$.  Theorem \ref{Zhang inequality} implies that $e_1(\Dbar_i)\geq 0$, for $i = 1,2$, because $\Dbar_i$ is normalized.  Therefore, we also have $e_1(\Dbar_1 + \Dbar_2) \geq 0$ for the essential minimum of the sum $h_{\Dbar_1} + h_{\Dbar_2}$.  The existence of the sequence $\{t_n\}$ thus implies that $e_1(\Dbar_1 + \Dbar_2) = 0$.  As $\Dbar_1 \cdot \Dbar_2 \geq 0$ from Theorem \ref{isomorphic} and $\Dbar_i \cdot \Dbar_i = 0$ for $i = 1,2$ by assumption, we apply Zhang's inequality (Theorem \ref{Zhang inequality}) to $\Dbar_1 + \Dbar_2$ to obtain  
	$$0 = e_1 \left( \Dbar_1 + \Dbar_2 \right) \; \geq \;  \frac{2\,  \Dbar_1 \cdot \Dbar_2 }{\deg D_1 + \deg D_2} \geq 0,$$
which allows us to deduce condition $(1)$.
\end{proof}
	
We will use the equivalences of Theorem \ref{isomorphic+} repeatedly in our proofs of Theorems \ref{Lambda} and \ref{scheme}.

	
\bigskip
\section{A metrized $\R$-divisor for each element of $E(k)\otimes \R$}
\label{R points}
	
	Throughout this section, we let $\mcE \to B$ be a non-isotrivial elliptic surface defined over a number field $K$, and let $E$ be the corresponding elliptic curve over the field $k = \Kbar(B)$.  We denote the zero by $O \in E(k)$. As $E(k)$ is finitely generated, we enlarge $K$ if needed so that all sections of $\mcE \to B$ are defined over $K$.  Recall that points $P_1, \ldots, P_m \in E(k)$ are {\bf independent} if the relation 
	$$a_1 P_1 + \cdots + a_m P_m = O$$
in $E(k)$ with $a_i \in \bZ$ implies that $a_1 = \cdots = a_m = 0$.

	In this section, we show that each nonzero element $X \in E(k)\otimes \bR$ naturally gives rise to an adelic, semipositive continuous metrization $\Dbar_X$ associated to an ample $\bR$-divisor $D_X$ on $B$; see Theorem \ref{real point divisor theorem}.  For $P \in E(k)$, these metrizations on $\R$-divisors coincide with the adelically metrized line bundles on $B$ that we studied in \cite{DM:variation}.  In \S\ref{bilinearity}, we observe that the assignment $X \mapsto \Dbar_X$ is quadratic, in the sense that $\Dbar_X \iso \Dbar_{\<X,X\>}$ for a bilinear operator $(X,Y) \mapsto \Dbar_{\<X,Y\>} := \frac12 (\Dbar_{X+Y} - \Dbar_X - \Dbar_Y)$ on $E(k)\otimes\R$.  
	
We begin by recalling the basic properties of N\'eron-Tate heights and their local decompositions.
	
\subsection{N\'eron-Tate heights}
Let $\mathcal{F}$ be a number field or a function field of transcendence degree 1 in characteristic 0.  We let $M_{\mathcal{F}}$ denote its set of places.  Let $E/\mathcal{F}$ be an elliptic curve with origin $O$, expressed in Weierstrass form as 
	$$E = \{ y^2+a_1xy+a_3y=x^3+a_2x^2+a_4x+a_6\}$$ 
with discriminant $\Delta$.  Denote by 
	$$\hat{h}_{E}:E(\overline{\mathcal{F}})\to[0,\infty)$$
a N\'eron-Tate canonical height function; it can be defined by 
	$$\hat{h}_E(P) = \frac12 \lim_{n\to\infty} \frac{h(x(nP))}{n^2}$$
where $h$ is the naive Weil height on $\pr^1$ and $x: E \to \pr^1$ is the degree 2 projection to the $x$-coordinate.  
	
For each $v\in M_{\mathcal{F}}$, recall that $\mathcal{F}_v$ denotes the completion of $\mathcal{F}$ with respect to $|\cdot|_v$ and $\bC_v$ denote the completion of the algebraic closure of $\mathcal{F}_v$.  The canonical height has a decomposition into local heights, as 
	$$\hat{h}_E(P)=  \frac{1}{|\Gal( \overline{\mathcal{F}}/\mathcal{F})\cdot P|}  \;
	\sum_{Q \in \Gal( \overline{\mathcal{F}}/\mathcal{F})\cdot P}    \; 
	\sum_{v\in M_{\mathcal{F}}} r_v \, \hat{\lambda}_{E,v}(Q)$$ 
for all $P\in E(\overline{\mathcal{F}})\setminus\{O\}$, with $r_v$ defined by \eqref{rv} in the number field case, and $r_v=1$ for function fields.  The local heights $\hat{\lambda}_{E,v}$ are characterized by the three properties \cite[Chapter 6, Theorem 1.1]{Silverman:Advanced}:
\begin{enumerate}
	\item	 $\hat{\lambda}_{E,v}$ is continuous on $E(\bC_v)\setminus \{O\}$ and bounded on the complement of any $v$-adic neighborhood of $O$;
	\item the limit of $\hat{\lambda}_{E,v}(P) - \frac12 \log|x(P)|_v$ exists as $P \to O$ in $E(\bC_v)$; and
	\item for all $P = (x,y) \in E(\bC_v)$ with $2\, P \not= O$, 
	\begin{align}\label{duplication}
		\hat{\lambda}_{E,v}(2\,P) = 4 \hat{\lambda}_{E,v}(P) -  \log|2y+a_1x+a_3|_v + \frac{1}{4} \log|\Delta|_v.
	\end{align}
\end{enumerate}
Property $(3)$ may be replaced with the quasi-parallelogram law
\begin{align}\label{quasiparallelogram}
\begin{split} \hat{\lambda}_{E,v}(P+Q) + \hat{\lambda}_{E,v}(P-Q) &= 2 \hat{\lambda}_{E,v}(P) + 2 \hat{\lambda}_{E,v}(Q) \\
		&- \log|x(P) - x(Q)|_v +\frac{1}{6} \log|\Delta|_v
\end{split}
\end{align}
under the assumption that none of $P$, $Q$, $P+Q$, nor $P-Q$ is equal to $O$.
Note that $\hat{\lambda}_{E,v}$ is independent of the choice of Weierstrass equation for $E$ over $\mathcal{F}$. It is useful to recall also the triplication formula; if $3P\neq O$, then
\begin{align}\label{triplication}
	\hat{\lambda}_{E,v}(3P) = 9 \hat{\lambda}_{E,v}(P) -  \log|(3x^4+b_2x^3+3b_4x^2+3b_6x+b_8)(P)|_v - \frac{2}{3} \log|\Delta|_v,
\end{align}
where $b_i$ are the usual Weierstrass quantities; see e.g. \cite[pg. 463]{Silverman:Advanced}.

\subsection{Metrized divisors for elements of $E(k)$}
	
Fix non-torsion $P\in E(k)$.  Define 
	$$D_P:=\sum_{\gamma\in B(\Kbar)}\hat{\lambda}_{E,\ord_{\gamma}}(P) \, [\gamma].$$
We remark that $\hat{\lambda}_{E,\ord_{\gamma}}(P)\in\bQ$ \cite[Chapter 11, Theorem 5.1]{Lang:Diophantine}, so $D_P$ is a $\mathbb{Q}$-divisor on $B$.  As $P$ is defined over $K$, the divisor is $\Gal(\Kbar/K)$-invariant.
	
In \cite[Theorem 1.1]{DM:variation} we established that $D_P$ can be equipped with an adelic, semipositive, continuous and normalized metrization 
\begin{equation} \label{Q-divisor}
	\overline{D}_P:=(D_P,\{\lambda_{P,v}\}_{v\in M_K})
\end{equation}
over the number field $K$, where $\lambda_{P,v}$ denotes the extension of $t\mapsto\hat{\lambda}_{E_t,v}(P_t)$ to $B^{\an}_v$.  It follows that the associated height functions satisfy 
	$$h_P(t) := h_{\Dbar_P}(t) = \hat{h}_{E_t}(P_t)$$
for all $t \in B(\Kbar)$ for which $E_t$ is smooth.  Both minima $e_1(\Dbar_P)$ and  $e_2(\Dbar_P)$ (defined in \S\ref{minima}) are equal to 0 \cite[Proposition 4.3]{DM:variation}; this allowed us to conclude that $\Dbar_P \cdot \Dbar_P = 0$ from Theorem \ref{Zhang inequality}.
	
For $O \in E(k)$, we set 
	$$\Dbar_O := ( 0, 0),$$
the trivial divisor with all functions $g_v = 0$. For torsion points $T\not= O \in E(k)$, the metrized divisor $\Dbar_T$ can also be defined by \eqref{Q-divisor}, with $\lambda_{T,v}(t) := \hat{\lambda}_{E_t,v}(T_t)$ for all $t \in B(\Kbar)$ with $E_t$ smooth.   The following proposition is key for the passage from $E(k)$ to $E(k)\otimes\R$.

\begin{proposition}\label{keyiso}
The metrized divisor $\Dbar_P$ is well defined for $P$ in $E(k)/E(k)_{\mathrm{tors}}$, up to isomorphism.  Moreover, for each $m\geq 1$ and any set of independent points $P_1,\ldots,P_m\in E(k)$ and integers $a_1,\ldots, a_m$, the following metrized divisors are isomorphic: 	
\begin{equation} \label{line bundle isomorphism}
	\overline{D}_{a_1P_1+\cdots+ a_mP_m} \iso \sum_{i=1}^{m}\left( a_i^2-a_i\sum_{j\not=i}a_j \right)\overline{D}_{P_i}+\sum_{1\le i<j\le m}a_ia_j\overline{D}_{P_i+P_j}
\end{equation}
\end{proposition}

\begin{remark}
The proposition implies, in particular, that the functions 
	$$t\mapsto \sum_{i=1}^{m} \left( a_i^2-a_i\sum_{j\not=i} a_j \right){\lambda}_{P_i,v}(t) +\sum_{1\le i<j\le m}a_ia_j \, {\lambda}_{P_i+P_j, v}(t)$$
are subharmonic on $B_v^{\an}$ (away from the points $t$ where $\lambda_{P_i,v}(t)$ or  ${\lambda}_{P_i+P_j, v}(t)$ is equal to $\infty$), for all choices of $a_i \in \bZ$, and at every place $v$ of $K$.
\end{remark}

We begin with a lemma (c.f. \cite[Exercise 6.4]{Silverman:Advanced}):
	
\begin{lemma}  \label{multiples}
Fix $P \in E(k)$.  For each nonzero $m\in\bZ$ with $|m|\ge 2$ such that $mP \not= O$, there exist $h \in K(B)$ and constant $c \in \bQ$ so that 
	$$\hat{\lambda}_{E_t,v}(mP_t)=m^2\hat{\lambda}_{E_t,v}(P_t)+ c  \log|h(t)|_v$$
at every place $v$ and for each $t\in B(\Kbar)$ such that $E_t$ is smooth.  If $P$ is torsion of order $m \geq 2$, we have 
	$$\hat{\lambda}_{E_t,v}(P_t) = c \log|h(t)|_v,$$
for some $c \in\bQ$ and $h\in K(B)$.
\end{lemma}

\begin{proof}
Upon replacing $P$ by $-P$, it suffices to prove the statement for $m\ge 2$.  The duplication formula \eqref{duplication} provides the desired result for $m=2$, assuming that $2P \not=O$.  Now fix any $m\geq 3$ and $P\in E(k)$, and assume that $mP \not= O$, $(m-1)P\neq O$ and $(m-2)P\neq O$.   Then the quasi-parallelogram law \eqref{quasiparallelogram} implies
\begin{eqnarray} \label{m cases}
	\hat{\lambda}_{E_t,v}(mP_t)&=& 2 \, \hat{\lambda}_{E_t,v}((m-1)P_t)+2\, \hat{\lambda}_{E_t,v}(P_t)-\hat{\lambda}_{E_t,v}((m-2)P_t) \\
		&& \quad  - \, \log\left| x((m-1)P_t)-x(P_t)\right|_v +\frac{1}{6} \log|\Delta_t|_v, \nonumber
\end{eqnarray}
for each $t\in B(\mathbb{C}_v)$ such that $E_t$ is smooth and $mP_t \not= O_t$, $(m-1)P_t\neq O_t$ and $(m-2)P_t\neq O_t$ and therefore for all $t\in B(\bC_v)$ by the continuity of the local height $t\mapsto\hat{\lambda}_{E_t,v}(P_t)\in\mathbb{R}\cup\{\pm\infty\}$.  The desired relation, for all non-torsion points and for all $m\geq 3$, then follows from \eqref{m cases} by an easy induction.  
		
Now suppose that $2P = O$ with $P\not= O$.  Then we have $3P=P\neq O$, and the triplication formula \eqref{triplication} implies that 
	$$\hat{\lambda}_{E_t,v}(3P_t) = 9 \hat{\lambda}_{E_t,v}(P_t) -  c\log|h(t)|_v=\hat{\lambda}_{E_t,v}(P_t),$$
for a constant $c\in\mathbb{Q}$ and $h\in K(B)$ and for all but finitely many $t$.  The equation then holds for all $t\in B(\bC_v)$ by the continuity of the local heights, and it implies that $\hat{\lambda}_{E_t,v}(P_t) = \frac{c}{8} \log |h(t)|_v$.  
		
For a torsion point $P$ of order $3$, we have $2P = -P \not= O$, so we may apply the duplication formula \eqref{duplication} to see that 
	$$\hat{\lambda}_{E_t,v}(2P_t) = 4 \hat{\lambda}_{E_t,v}(P_t) -  c\log|h(t)|_v=\hat{\lambda}_{E_t,v}(-P_t) = \hat{\lambda}_{E_t,v}(P_t),$$
for a constant $c\in\mathbb{Q}$ and $h\in K(B)$.  It follows that $\hat{\lambda}_{E_t,v}(P_t) = \frac{c}{3} \log |h(t)|_v$.
		
Finally, suppose that $P$ is torsion of order $n\geq 4$, and note that $(n-1)P = -P\neq O$, $(n-2)P\neq O$ and $(n-3)P\neq O$.  We infer from \eqref{m cases} with $3 \leq m \leq n-1$ inductively that 
	$$\hat{\lambda}_{E_t,v}((n-1)P_t)= (n-1)^2\hat{\lambda}_{E_t,v}(P_t) - c \log|h(t)|_v = \hat{\lambda}_{E_t,v}(-P_t)=\hat{\lambda}_{E_t,v}(P_t)$$
for a rational function $h \in K(B)$ and $c \in \bQ$, so that $\hat{\lambda}_{E_t,v}(P_t) = \frac{c}{n^2-2n} \log|h(t)|_v.$
\end{proof}
	
\begin{proof}[Proof of Proposition \ref{keyiso}]
Lemma \ref{multiples} implies that 
	$$\Dbar_P \iso \Dbar_O$$
for every torsion point $P \in E(k)$.  Furthermore, for any non-torsion point $P$, Lemma \ref{multiples} also implies that 
	$$\Dbar_{aP} \iso a^2 \Dbar_P$$
for all $a \in \bZ$, demonstrating \eqref{line bundle isomorphism} for $m=1$.  Therefore, if $P$ is non-torsion and $Q$ is torsion of order $n\geq 2$, we have 
	$$\Dbar_{P+Q} \iso \frac{1}{n^2} \Dbar_{n(P+Q)} = \frac{1}{n^2} \Dbar_{nP} \iso \Dbar_P.$$
This proves that the metrized divisors depend only on the class in $E(k)/E(k)_{\mathrm{tors}}$, up to isomorphism.
		
Now fix any $m\geq 2$, and any collection of independent points $P_1, \ldots, P_m \in E(k)$ and integers $a_1, \ldots, a_m$.  Define a divisor on $B$ by
	$$D' = \sum_{i=1}^{m}\left( a_i^2-a_i\sum_{j\not=1}a_j \right)D_{P_i}+\sum_{1\le i<j\le m}a_ia_j \, D_{P_i+P_j},$$
and consider the metrization on $D'$ defined by 
	$$g_v(t) = \sum_{i=1}^{m} \left( a_i^2-a_i\sum_{j\not=1} a_j \right)\lambda_{P_i, v}(t) +\sum_{1\le i<j\le m}a_ia_j \, \lambda_{P_i+P_j, v}(t).$$
To prove the proposition, we will use the quasi-parallelogram law \eqref{quasiparallelogram} to show that there exists a rational function $f \in K(B)$ so that 
\begin{equation} \label{subharmonic equality}
	g_v(t) - \hat{\lambda}_{E_t,v}(a_1P_{1,t}+\cdots+ a_mP_{m,t}) =  \log|f(t)|_v
\end{equation}
at all places $v$ of $K$ and for all but finitely many $t\in B(\bC_v)$.  
		
\begin{lemma}\label{sumsofthree}
Let $P,Q,R\in E(k)$ be independent points defined over $K$.  Then, there is a rational function $f_{P,Q,R}\in K(B)$ such that 
\begin{eqnarray*}
	\hat{\lambda}_{E_t,v}(P_t+Q_t+R_t)&=& \hat{\lambda}_{E_t,v}(P_t+R_t)+\hat{\lambda}_{E_t,v}(P_t+Q_t)+\hat{\lambda}_{E_t,v}(Q_t+R_t)\\
		&&  \quad -\hat{\lambda}_{E_t,v}(P_t)-\hat{\lambda}_{E_t,v}(Q_t)-\hat{\lambda}_{E_t,v}(R_t)-\log|f_{P,Q,R}(t)|_v.
\end{eqnarray*}		
for all $t\in B(\overline{K})$ such that $E_t$ is smooth and all $v\in M_K$. 
\end{lemma}	
		
\begin{proof}
The proof follows by applying the quasi-parallelogram law \eqref{quasiparallelogram} for the pairs $\{P+R, Q\}$, $\{P,R-Q\}$, $\{P+Q, R\}$ and $\{R,Q\}$ and taking an alternating sum as in \cite[Theorem 9.3]{Silverman:Elliptic}. 
\end{proof}
		
\begin{lemma}\label{quasi sum}
Fix independent $P, Q\in E(k)$.  For each $(a,b)\in\mathbb{Z}^2\setminus\{(0,0)\}$, there is rational function $h_{a,b}\in K(B)$ such that
\begin{eqnarray*}		
	\hat{\lambda}_{E_t,v}(aP_t+bQ_t)&=&
		(a^2-ab)\hat{\lambda}_{E_t,v}(P_t)+ab  \hat{\lambda}_{E_t,v}(P_t+Q_t)+\\
		&& \quad (b^2-ab)\hat{\lambda}_{E_t,v}(Q_t)-\log|h_{a,b}|_v,
\end{eqnarray*}
for all $t\in B(\overline{K})$ such that $E_t$ is smooth and all $v\in M_K$. 
\end{lemma}
		
\begin{proof}
The proof follows from the quasi-parallelogram law by an easy induction.  Lemma \ref{multiples} provides the desired result if one of $a$ or $b$ is 0.  Next we will show that for each $n\in\bZ$ there is a rational function $g\in K(B)$ so that 
\begin{align}\label{multipleplus}
	\begin{split}
	\hat{\lambda}_{E_t,v}(nP_t+Q_t)&=(n^2-n)\hat{\lambda}_{E_t,v}(P_t)+n  \hat{\lambda}_{E_t,v}(P_t+Q_t)\\
	&\quad +(1-n)\hat{\lambda}_{E_t,v}(Q_t)-\log|g|_v.
		\end{split}
\end{align}
Replacing $P$ by $-P$ we may assume that $n\ge 1$.  For $n=1$ the statement is clear. For $n\ge 1$, the quasi-parallelogram law \eqref{quasiparallelogram} implies that
\begin{eqnarray*}
	\hat{\lambda}_{E_t,v}((n+1)P_t+Q_t)&=&  \hat{\lambda}_{E_t,v}(nP_t+(P+Q)_t) \\
	&=&2\hat{\lambda}_{E_t,v}(nP_t)+2\hat{\lambda}_{E_t,v}(P_t+Q_t)-\hat{\lambda}_{E_t,v}((n-1)P_t-Q_t)\\
	&&  \quad -\log|x(nP_t)-x(P_t+Q_t)|_v  +\frac{1}{6} \log|\Delta|_v
\end{eqnarray*}
and \eqref{multipleplus} follows inductively from Lemma \ref{multiples}. Using \eqref{multipleplus} we now have a rational function $h\in K(B)$ so that
\begin{eqnarray}
	\hat{\lambda}_{E_t,v}(aP_t+bQ_t)&=& (a^2-a)\hat{\lambda}_{E_t,v}(P_t)+a\hat{\lambda}_{E_t,v}(P_t+bQ_t)+\\
	&&  \quad (1-a)\hat{\lambda}_{E_t,v}(bQ_t)-\log|h|_v.  \nonumber
\end{eqnarray}
The lemma then follows by another application of \eqref{multiples} and \eqref{multipleplus}, exchanging the roles of $P$ and $Q$.  
\end{proof}
		
Finally, a simple induction using Lemmas \ref{sumsofthree} and \ref{quasi sum} implies that for any $m\ge 2$, and for any integers $a_1, \ldots, a_m$, the equality \eqref{subharmonic equality} holds, for some rational $f$.  This completes the proof of the proposition.
\end{proof}

\subsection{Metrized divisors for elements of $E(k)\otimes\bR$}  \label{R divisor subsection}
Fix nonzero $X \in E(k)\otimes \bR$.   Choose independent points $P_1, \ldots, P_m \in  E(k)$ that define a basis for $E(k)\otimes \bR$, and write $X = x_1P_1 + \cdots + x_m P_m$ with $x_i \in \R$.  With a slight abuse of notation, we identify the two isomorphic metrized divisors in Proposition \ref{keyiso} and define an adelically metrized $\bR$-divisor on $B(\Kbar)$, over the number field $K$, by    
\begin{align}  \label{real point divisor}
	\Dbar_{X} := \sum_{i=1}^{m} \bigg( x_i^2-x_i\sum_{\substack{j\neq i}}x_j \bigg)\overline{D}_{P_i}+\sum_{1\le i<j\le m}x_ix_j \; \overline{D}_{P_i+P_j}
\end{align}
for the $\Dbar_P$ defined by \eqref{Q-divisor} when $P \in E(k)$. It defines a height function 
\begin{align}\label{associated height}
	h_X(t) = \sum_{i=1}^{m} \left( x_i^2-x_i\sum_{j\not=i} x_j \right)\hat{h}_{E_t}(P_{i,t}) + \sum_{1\le i<j\le m}x_ix_j \, \hat{h}_{E_t}(P_{i,t}+P_{j,t})
\end{align}
at all points $t \in B(\Kbar)$ for which $E_t$ is smooth.

\begin{theorem} \label{real point divisor theorem}
Fix nonzero $X \in E(k)\otimes \bR$.   The metrized divisor $\Dbar_X$ of \eqref{real point divisor} is continuous, adelic, semipositive and normalized.  The degree of the underlying $\R$-divisor $D_X$ is $\hat{h}_E(X) >0$.  Its associated height function satisfies
\begin{align}\label{real height equality}
	h_X(t)=\hat{h}_{E_t}(X_t)
\end{align}
for all $t \in B(\Kbar)$ with smooth fiber $E_t$.  Further, up to isomorphism, $\Dbar_X$ is independent of the choice of basis for $E(k)$.  
\end{theorem}

\begin{proof}
Fix $x_1,\ldots,x_m\in\bR$ and choose sequences of rational numbers $a_{n,i}/a_{n,0}\to x_i$ for $i=1,\ldots,m$.  From Proposition \ref{keyiso} we know that the functions 
	$$\frac{1}{a_{n,0}^2}\left(\sum_{i=1}^{m} \bigg( a_{n,i}^2-a_{n,i}\sum_{\substack{j\neq i}}a_{n,j} \bigg)\lambda_{P_j, v}(t)+ \sum_{1\le i<j\le m}a_{n,i}a_{n,j}\lambda_{P_i+P_j, v}(t)\right)$$
are continuous, subharmonic functions on $B_v^{\an}$ (away from their logarithmic singularities), because they define a metrized divisor isomorphic to $a_{n,0}^{-2} \; \Dbar_{a_{n,1} P_1 + \cdots + a_{n,m} P_m}$.  
The limit as $n\to \infty$ clearly exists as a continuous, semipositive, adelic metrization on an $\R$-divisor
	$$D_X =   \sum_{i=1}^{m}\bigg( x_i^2-x_i\sum_{j\not=1}x_j \bigg)D_{P_i}+ \sum_{1\le i<j\le m}x_ix_j \, D_{P_i+P_j}.$$
To see that $\Dbar_X$ is normalized, recall that by \cite[Theorem 1.1]{DM:variation} we have 
	$$\Dbar_{a_{n,1}P_1+\cdots+a_{n,m}P_m}\cdot \Dbar_{a_{n,1}P_1+\cdots+a_{n,m}P_m}=0,$$
for all $n\in\bN$. In view of Proposition \ref{keyiso} we then have
	$$\frac{1}{a_{n,0}^4}\left( \sum_{i=1}^{m} \bigg( a_{n,i}^2-a_{n,i}\sum_{\substack{ j\neq i}}a_{n,j} \bigg)\Dbar_{P_j}+ \sum_{1\le i<j\le m}a_{n,i}a_{n,j}\Dbar_{P_i+P_j}\right)^2 =0,$$
for all $n\in\bN$.  Letting $n\to\infty$ we get $\Dbar_X\cdot \Dbar_X=0$.  
		
Equation \eqref{real height equality} follows from the properties of $\hat{h}_{E_t}$ as a quadratic form on each smooth fiber $E_t$.  Specifically, we have
	$$\hat{h}_{E_t}(P_t + Q_t) = \hat{h}_{E_t}(P_t) + 2\< P_t, Q_t\>_{E_t} + \hat{h}_{E_t}(Q_t)$$
for the N\'eron-Tate bilinear form $\<P_t, Q_t\>_{E_t}$ and for any pair of points $P,Q\in E(k)$ and $t \in B(\Kbar)$ with $E_t$ smooth.   It follows that 
\begin{eqnarray*}
	\hat{h}_{E_t}(yP_t + zQ_t) 
	&=& y^2\hat{h}_{E_t}(P_t) + yz(\hat{h}_{E_t}(P_t + Q_t) - \hat{h}_{E_t}(P_t) -  \hat{h}_{E_t}(Q_t)) + z^2\hat{h}_{E_t}(Q_t)  \\
	&=& (y^2-yz)\hat{h}_{E_t}(P_t) + yz \, \hat{h}_{E_t}(P_t + Q_t)+ (z^2-yz)\hat{h}_{E_t}(Q_t) 
\end{eqnarray*}
for all $y,z \in \bR$.  Therefore, by induction, we deduce that 
\begin{eqnarray*}
	\hat{h}_{E_t}(x_1P_{1,t}+\cdots+ x_mP_{m,t}) &=&  \sum_{i=1}^m x_i^2 \, \hat{h}_{E_t}(P_{i,t}) + \, 2\sum_{i < j} x_i x_j \< P_{i,t}, P_{j,t}\>_{E_t} \\
	&=&  \sum_{i=1}^{m} \left( x_i^2-x_i\sum_{j\not=1} x_j \right)\hat{h}_{E_t}(P_{i,t}) + \sum_{1\le i<j\le m}x_ix_j \, \hat{h}_{E_t}(P_{i,t}+P_{j,t})
\end{eqnarray*}
for any collection $P_1, \ldots, P_m \in E(k)$ and real numbers $x_1, \ldots, x_m$, so that
	$$h_X(t)=\hat{h}_{E_t}(X_t)$$
for all $t \in B(\Kbar)$ with $E_t$ smooth.  
That $\Dbar_X$ does not depend on its presentation or the choice of basis follows easily from Proposition \ref{keyiso}.
\end{proof}
	
\subsection{Bilinearity}  \label{bilinearity} 
For $X,Y \in E(k)\otimes \bR$, consider the metrized $\R$-divisor
\begin{align}\label{inner product}
	\Dbar_{\<X,Y\>} := \frac12 \left(\Dbar_{X+Y} - \Dbar_X - \Dbar_Y\right)
\end{align}
on the base curve $B$, of degree equal to the N\'eron-Tate inner product of $X$ and $Y$,
	$$\<X,Y\>_E = \frac12 \left( \hat{h}_E(X+Y) - \hat{h}_E(X) - \hat{h}_E(Y)\right).$$
Note that $\Dbar_{\<X,Y\>}$ is symmetric in $X, Y \in E(k) \otimes \R$.  It is also bilinear, in the sense that	
\begin{eqnarray}  \label{bilinear divisor}
	\Dbar_{\<X,aY+bZ\>}
	&=&   \frac12 \left(  \Dbar_{X+aY+bZ} - \Dbar_X - \Dbar_{aY+bZ}  \right)  \\
	&\iso&  \frac12 \bigg( (1-a-b)\Dbar_X + (a^2-a-ab)\Dbar_Y + (b^2-b-ab)\Dbar_Z   \nonumber \\
			&& + \; a\Dbar_{X+Y} + b\Dbar_{X+Z} + ab\Dbar_{Y+Z}   - \Dbar_X   \nonumber \\  
		&&  - \;   \left[ (a^2-ab)\Dbar_Y + (b^2-ab)\Dbar_Z + ab \Dbar_{Y+Z} \right]\bigg)  \nonumber \\
			&=&  \frac12 \left( a\left[\Dbar_{X+Y}-\Dbar_X-\Dbar_Y\right] + b \left[ \Dbar_{X+Z}-\Dbar_X-\Dbar_Z\right]  \right) \nonumber \\
		&=&  a \, \Dbar_{\<X,Y\>} + b\,    \Dbar_{\<X,Z\>},  \nonumber
\end{eqnarray}
from \eqref{real point divisor} and Theorem \ref{real point divisor theorem}.  Moreover, we have 
	$$\Dbar_X \iso \Dbar_{\<X,X\>}$$
for all $X \in E(k) \otimes \R$.

	\bigskip
	\section{Small sequences}  
	\label{small section}
	
	As before, we let $\mcE \to B$ be a non-isotrivial elliptic surface defined over a number field $K$, and let $E$ be the corresponding elliptic curve over the field $k = \Kbar(B)$.  In \S\ref{R divisor subsection}, we constructed metrized $\R$-divisors $\Dbar_X$ and associated height functions $h_X$ for each element $X \in E(k) \otimes \R$.  In this section, we look at the sets of ``small" points for the height $h_X$.   We conclude the section with a proof of Theorem \ref{geometry of relations}.

	\subsection{Small sequences exist}
	For an adelic, continuous, semipositive, and normalized metrized $\R$-divisor $\Dbar$ with ample $D$ on the curve $B$, an infinite sequence $\{t_n\} \subset B(\Kbar)$ is said to be {\bf small} if 
	$$h_{\Dbar}(t_n) \to 0$$
	as $n\to \infty$.

	\begin{proposition}  \label{small points exist}
		For every nonzero $X \in E(k) \otimes \bR$, there exist small sequences for $\Dbar_X$, so that the essential minimum is $e_1(\Dbar_X) = 0$.  
		More precisely, write $X = x_1 P_1 + \cdots + x_m P_m$ for $x_i\in \R$ and independent $P_i \in E(k)$, and choose integers $a_{i,n}$ for $i = 1, \ldots, m$ and $n\in\bN$ so that $[a_{1,n}:\cdots:a_{m,n}]\to [x_1:\cdots:x_m]$ as $n\to\infty$ in the real projective space $\mathbb{RP}^{m-1}$. 
Then there exists an infinite non-repeating sequence of points $t_n \in B(\Kbar)$ at which 
	$$(a_{1,n} P_1 + \cdots + a_{m,n} P_m)_{t_n}$$
is torsion in the fiber $E_{t_n}(\Kbar)$. Moreover, for any such sequence $\{t_n\}\subset B(\Kbar)$ we have
	$$h_X (t_n) \to 0$$
as $n\to \infty$.
\end{proposition}
	
	To prove Proposition \ref{small points exist}, we begin with a well-known statement that follows from Silverman's specialization theorem \cite{Silverman:specialization}.
	
\begin{lemma} \label{bounded height specialization}
		Fix any set of independent points $P_1, \ldots, P_m$ in $E(k)$, and let $h$ be any Weil height function on $B$ associated to a divisor of degree 1.  The set of all $t$ for which there exist integers $a_1, \ldots, a_m$, not all zero, such that 
		$$a_1 P_{1,t}  + \cdots + a_m P_{m,t} = O_t$$
		in $E_t$ has bounded $h$-height.
\end{lemma}

\begin{proof}
For each non-torsion point $Q \in E(k)$ we have \cite{Silverman:specialization}
	$$\lim_{h(t) \to \infty} \frac{\hat{h}_{E_t}(Q_{t})}{h(t)} = \hat{h}_E(Q) > 0,$$
so the set $\{t \in B(\Kbar):  \hat{h}_{E_t}(Q_t) = 0\}$ has bounded $h$-height.  Because $\det(\<P_i, P_j\>_E)_{i,j} >0$, it follows that the set 
	$$\mathcal{R}(P_1, \ldots, P_m) = \{t \in B(\Kbar):  \det ( \< P_{i,t}, P_{j,t}\>_{t} ) = 0\}$$
also has bounded height.  This set $\mathcal{R}(P_1, \ldots, P_m)$ contains the set of $t$ at which the points become linearly dependent. 
\end{proof}

\begin{proof}[Proof of Proposition \ref{small points exist}]
Write $X = x_1 P_1 + \cdots + x_m P_m$ for independent $P_1, \ldots, P_m \in E(k)$ and $x_1, \ldots, x_m \in \R$.   Fix a sequence of positive integers $M_n \to \infty$ as $n\to \infty$.  For $i = 1, \ldots, m$, choose any sequence of integers $a_{i,n}$ so that $a_{i,n}/M_n  \to  x_i$ as $n \to \infty$, and set 
	$$Q_n = a_{1,n} P_{1} + ... + a_{m,n} P_{m} \in E(k),$$
so that $\frac{1}{M_n} Q_n \to X$ in $E(k) \otimes\R$.

		Consider the set 
		$$\mathrm{Tor}(Q_n) = \{t \in B(\Kbar):  Q_{n,t} \mbox{ is torsion in } E_t\}$$
		For each $n$, the set $\Tor(Q_n)$ is infinite; in fact, it is dense in $B(\C)$ \cite[Proposition 6.2]{DM:variation} \cite[\S III.2 and Notes to Chapter III]{Zannier:book}. Moreover, from Lemma \ref{bounded height specialization}, this set has bounded height in the base curve $B$ with respect to any chosen Weil height $h$, and the height is bounded independent of $n$.  Therefore, from  \cite[Theorem A]{Silverman:specialization}, we can find $H>0$ so that 
		$$h_{P_i}(t) \leq H \quad\mbox{ and } \quad h_{P_i + P_j}(t) \leq H$$
		for all $t \in \bigcup_n \Tor(Q_n)$ and for all $i, j$. 
		
		From the formula for the height $h_X$ given in \eqref{associated height} and the formula for the height of $Q_n$ appearing in Proposition \ref{keyiso}, we have the following.  For any given $\eps> 0$, there exists $N>0$ so that 
		\begin{eqnarray*}  
			\left| h_X(t) - \frac{1}{M_n^2} h_{Q_n}(t) \right| 
			&=&  \left|   \sum_{i=1}^{m} \left( x_i^2-x_i\sum_{j\not=1} x_j  - \frac{a_{i,n}^2}{M_n^2} + \frac{a_{i,n}}{M_n}\sum_{j\not=i}  \frac{a_{j,n}}{M_n}  \right)  h_{P_i}(t) \right. \\
			&& \qquad\qquad\qquad  	\left. + \;  \left(  \sum_{1\le i<j\le m}x_ix_j - \frac{a_{i,n}a_{j,n}}{M_n^2}\right)  h_{P_i + P_j}(t) \right|   < \eps
		\end{eqnarray*}
		for all $n > N$ and for all $t$ where $h_{P_i}(t)\leq H$ and $h_{P_i + P_j}(t) \leq H$ for all $i, j$.  In particular, the estimate holds for all $t \in \bigcup_{n\geq 1} \Tor(Q_n)$.
		
		For each $n$ and for every $t \in \Tor(Q_n)$, we have $h_{Q_n}(t) = 0$.  Choosing any sequence of distinct points $t_n \in B(\Kbar)$ so that $Q_{n,t_n}$ is torsion in $E_{t_n}$, we may conclude that
		$$h_X(t_n) \to 0$$
		as $n\to \infty$.  
	\end{proof}
	
	\subsection{Characterization of small sequences}  
	Here, we observe that small sequences for real points $X \in E(k)\otimes \R$ always arise from a construction similar to that of Proposition \ref{small points exist}, where relations between the generators are ``almost" satisfied.  We will use this next proposition in the proof of Theorem \ref{TFAE}.

	\begin{proposition} \label{regulatorvsreal}
		Let $M$ be a torsion-free subgroup of $E(k)$ of rank $m$, generated by $S_1,\ldots,S_{m}$.  Set $h_{M}(t) =\displaystyle\det(\<S_{i,t}, S_{j,t}\>_t)$, for the N\'eron-Tate bilinear form $\<\cdot,\cdot\>_t$ on the fiber $E_t(\Kbar)$. For a non-repeating infinite sequence $t_n\in B(\overline{K})$, the following are equivalent
		\begin{enumerate}
			\item  $\liminf_{n\to\infty} h_M(t_n)=0;$
			\item  there is a non-zero $X \in M\otimes\bR$ such that $\liminf_{n\to\infty} h_X(t_n)=0$.  
			\item  there are sequences of points $s_{i,n} \in E_{t_n}(\Kbar)$, for $i = 1, \ldots, m$, satisfying 
			$$\liminf_{n\to \infty} \left( \max_i \hat{h}_{E_{t_n}}(s_{i,n}) \right) = 0$$
			and so that the points 
			$$S_{1,t_n} - s_{1,n}, \; \ldots \; , \; S_{m,t_n} - s_{m,n}$$
			satisfy a linear relation over $\bZ$ in $E_{t_n}(\Kbar)$. 
		\end{enumerate}
	\end{proposition}
	
	This proposition relies heavily on Silverman's specialization results \cite[Theorem A and Theorem B]{Silverman:specialization}. We point out that \cite[Theorem B]{Silverman:specialization} holds for real points $X\in E(k)\otimes\mathbb{R}$ by the bilinearity of the N\'eron-Tate pairing. 
	We begin with a lemma. 
	
	\begin{lemma}  \label{one relation}
		Assume we are in the setting of Proposition \ref{regulatorvsreal}.  Assume further that there are sequences of points $s_{i,n} \in E_{t_n}(\Kbar)$, for $i = 1, \ldots, m$, satisfying 
		$$\sup_n\left( \max_i \hat{h}_{E_{t_n}}(s_{i,n}) \right) < \infty$$
		for which the points 
		$$S_{1,t_n} - s_{1,n}, \; \ldots \; , \; S_{m,t_n} - s_{m,n}$$
		satisfy a linear relation over $\bZ$ in $E_{t_n}(\Kbar)$.  Then the sequence $\{t_n\}$ will have bounded height in $B(\Kbar)$ with respect to any Weil height on $B$.
	\end{lemma}
	
	\begin{proof}
		Fix any Weil height $h$ on $B(\Kbar)$ of degree 1.  Consider the $m\times m$ matrix 
		$$
		A_n:= \left( \<S_{i,t_n}-s_{i,n}, \, S_{j,t_n}-s_{j,n}  \> _{t_n} \right)_{i,j} 
		$$
		where $\<\cdot, \cdot \>_{t_n}$ is the N\'eron-Tate inner product on the fiber $E_{t_n}(\Kbar)$.  Our assumption implies that 
		$$\det A_n =0$$
		for all $n$.  Assume that $h(t_n)\to \infty$. 
		Then by Silverman's specialization theorem \cite[Theorem B]{Silverman:specialization} we have 
		$$\frac{\<S_{i,t_n},S_{j,t_n}\>_{t_n}}{h(t_n)}\to \<S_i,S_j\>_E,$$
		as $n\to \infty$ for all $i,j=1,\ldots,m$.
		On the other hand, the bounded height of the perturbations $s_{i,n}$ and the Cauchy-Schwarz inequality for $\<\cdot,\cdot\>_{t_n}$ implies that 
		$$\left\lvert \frac{\<s_{i,n},s_{j,n}\>_{t_n}}{h(t_n)}\right\lvert\le \frac{\sqrt{ \hat{h}_{E_{t_n}}(s_{i,n}) \hat{h}_{E_{t_n}}(s_{j,n})  }}{h(t_n)}\to 0.$$
		Using Silverman's specialization \cite[Theorem A]{Silverman:specialization} we also have 
		$$\left\lvert \frac{\<S_{i,t_n},s_{j,n}\>_{t_n}}{h(t_n)}\right\lvert\le \frac{\sqrt{ \hat{h}_{E_{t_n}}(S_{i,t_n}) \hat{h}_{E_{t_n}}(s_{j,n})  }}{h(t_n)}\to 0.$$
		Combining these estimates, we obtain 
		$$0 = \frac{\det A_n}{(h(t_n))^m}\longrightarrow \det(\<S_i,S_j\>_E)_{i,j} \not=0,$$
		which is a contradiction.  
	\end{proof}

	\begin{proof}[Proof of Proposition \ref{regulatorvsreal}]
		Assume condition (2).  Let $X\in M\otimes\bR$ be nonzero and $\{t_n\}$ a sequence for which $\displaystyle\liminf_{n\to\infty} h_X(t_n)=0.$  Write $X=x_1S_{1}+\cdots+x_{\ell}S_{\ell}$, for  $x_i\in\bR$ not all equal to $0$.   After reordering the points $S_i$ we may assume that $x_1\neq 0$. Notice that 
		\begin{align*}	\det{\left(\begin{array}{cccc}
					\hat{h}_{E_{t_n}}(X_{t_n})& \<X_{t_n},S_{2,t_n}\>& \cdots & \<X_{t_n},S_{\ell,t_n}\>  \\
					\<S_{2,t_n},X_{t_n}\>  &  \hat{h}_{E_{t_n}}(S_{2,t_n}) &\cdots & \<S_{2,t_n},S_{\ell,t_n}\>\\
					\vdots & \vdots    &  & \vdots\\
					\<S_{\ell,t_n},X_{t_n}\> &  \<S_{\ell,t_n},S_{2,t_n}\>    &\cdots  & \hat{h}_{E_{t_n}}(S_{\ell,t_n})
				\end{array} \right)}&=x_1^2\, h_M(t_n),
		\end{align*}
		which easily follows by subtracting from the first column the sum of $x_i$ times the $j$-th column over all $j=2,\ldots,\ell$ and then subtracting from the first row the sum of $x_i$ times the $i$-th row over all $i=2,\ldots,\ell$.  Expanding the determinant along the first column we get 
		\begin{align}\label{firstrawdet}
			x_1^2\, h_M(t_n)=	h_X(t_n) \, f_{1,n}+\sum_{j=2}^{\ell}\<S_{j,t_n},X_{t_n}\>f_{j,n}, 
		\end{align}
		where for all $n\in\bN$ the $f_{j,n}$ are polynomial functions of the quantities $\<S_{j,t_n},X_{t_n}\>$ and $\<S_{j,t_n},S_{k,t_n}\>$ for $j,k=2,\ldots,\ell$. 
		Passing to a subsequence of $\{t_n\}$ we have $\displaystyle\lim_{n\to\infty} h_X(t_{k_n})=0$. In particular, since $X$ is non-trivial, \cite[Theorem B]{Silverman:specialization} yields that $\{h(t_{k_n})\}_{n\in\bN}$ is a bounded sequence. Using then \cite[Theorem A]{Silverman:specialization}, the functoriality of heights and the Cauchy-Schwarz inequality we get
		\begin{align}\label{thankstate}
			\max\{|f_{1,k_n}|,\ldots,|f_{\ell,k_n}|\}\le L, 
		\end{align}
		for some $L > 0$.  Moreover, for all $j=2,\ldots,\ell$ and all $n\in\bN$ we have 
		\begin{align}\label{cauchy-sch}
			|\<S_{j,t_{k_n}},X_{t_{k_n}}\>|^2 \; \le \; \hat{h}_{E_{t_{k_n}}}(S_{j,t_{k_n}})\hat{h}_{E_{t_{k_n}}}(X_{t_{k_n}}) \; \le\;  L\, h_X(t_{k_n})\to 0.
		\end{align}
		Our assumption on $X$ together with \eqref{firstrawdet}, \eqref{thankstate} and \eqref{cauchy-sch} yield
		$$\liminf_{n\to\infty} h_M(t_n)=0,$$
		proving that condition (1) holds.
		
		Now assume (1).  Let $A_t = (\<S_{i,t},S_{j,t}\>)_{i,j}$, so that $h_M(t) = \det A_t$, and consider the family of quadratic forms
		\begin{eqnarray*}	
			q_t(\vec z) &:=& \hat{h}_{E_t}(z_1S_{1,t}+\cdots z_{m} S_{m,t})\\
			&=&  \sum_{k=1}^{m}z_k^2 \, \hat{h}_{E_t}(S_{k,t})+ 2\sum_{ i <j }z_iz_j\<S_{i,t},S_{j,t}\>=\vec{z}\, A_t \, \vec{z}^\top,
		\end{eqnarray*}
		for $\vec{z}=(z_1,\ldots,z_m)\in\bR^m$, indexed by $t\in B(\overline{K})$ where $E_t$ is smooth.  Since $q_t \ge 0$ for all $t$, we have that $A_t$ has non-negative eigenvalues. Our assumption is that 
		$$\displaystyle\liminf_{n\to\infty}\, \det A_{t_n} = 0,$$
		so, if $\la_n$ is the smallest eigenvalue of $A_{t_n}$, then
		$$\displaystyle\liminf_{n\to\infty}\la_n= 0.$$
		Let $\displaystyle\vec{v}_n=(v_{1,n},\ldots,v_{m,n})\neq 0$ be an eigenvector of $A_{t_n}$ corresponding to $\la_n$.  Then 
		$$\vec{v}_n\, A_{t_n} \,\vec{v}_n^\top=\la_n||\vec{v}_n||^2,$$
		so that 
		\begin{align}\label{qn}
			\liminf_{n\to\infty}q_{t_n}\left(\frac{\vec{v}_n}{||\vec{v}_n||} \right)= \liminf_{n\to\infty}\la_n= 0.
		\end{align}
		Passing to a subsequence of the $\{t_n\}$, we have $\lim_{n\to\infty}h_M(t_n)=0$, and passing to a further subsequence, we may set 
		$$\vec{x}:=\lim_{n\to\infty}\frac{\vec{v}_{n}}{||\vec{v}_{n}||}\in\bR^m\setminus\{\vec{0}\}.$$ 
		By \cite[Theorem B]{Silverman:specialization}, the height of $\{t_n\}$ is bounded with respect to any choice of Weil height on $B$ (because $\det (\<S_i, S_j\>_E) \not=0$). 
		In view of \cite[Theorem A]{Silverman:specialization}, we get that the sequences $\{\<S_{i,t_n},S_{j,t_n}\>\}_{n}$ for $i,j=1,\ldots,\ell$ are bounded.  Thus \eqref{qn} yields
		$$\lim_{n\to\infty}q_{t_n}(\vec{x}) =0.$$ 
		In other words, for $X=x_1S_1+\cdots+x_mS_m$ we have $\lim_{n\to\infty} h_X(t_n)=0$, providing condition (2).
		
		Assuming (2), we now prove (3).  Reordering the points and rescaling $X$ if necessary, we may assume that $x_1 = 1$.  Passing to a subsequence, we have
		\begin{align}\label{X rewritten}
			\hat{h}_{E_{t_n}}(S_{1,t_n}+x_2S_{2,t_n} +\cdots+x_{m}S_{m,t_n})\to 0,
		\end{align}
		as $n\to \infty$.  
		Let $a_{2,n},\ldots, a_{m,n}$ be infinite sequences of integers satisfying $a_{i,n}/n\to x_i$ for each $i=2,\ldots,m$.   As $\hat{h}_E(X) \not=0$, we have by Silverman specialization \cite[Theorem B]{Silverman:specialization} that the sequence $\{t_n\}$ has bounded height in $B$.  Invoking \cite[Theorem A]{Silverman:specialization} we get that all sequences $\{\<S_{i,t_n},S_{j,t_n}\>_{E_{t_n}}\}_{n\in\bN}$ are bounded. Using the fact that each $\hat{h}_{E_{t_n}}(\cdot)$ defines a quadratic form on $E_{t_n}(\Kbar)$, line \eqref{X rewritten} yields
		\begin{equation} \label{sn small}	
			\hat{h}_{E_{t_n}}\left(S_{1,t_n}+\frac1n \left(a_{2,n}S_{2,t_n} +\cdots+a_{\ell,n}S_{\ell,t_n}\right)\right)\to 0.
		\end{equation}
		Since $\overline{K}$ is algebraically closed we may find $s_n\in E_{t_n}(\overline{K})$ so that 
		\begin{align}\label{sn}
			n\, s_n=a_{2,n}S_{2,t_n} +\cdots+a_{\ell,n}S_{\ell,t_n}.
		\end{align}
		Letting $s_{1,n}:=S_{1,t_n}+s_n$ and $s_{i,n} := O_{t_n}$ for all $i = 2, \ldots, m$, equation \eqref{sn small} yields 
		$$\hat{h}_{E_{t_n}}(s_{i,n})\to 0$$
		for each $i = 1, \ldots, m$.  Moreover by \eqref{sn} we have that the set $\{S_{1,t_n}-s_{1,n}, S_{2,t_n},\ldots, S_{\ell,t_n}\}$ is linearly dependent in $E_{t_n}$ for every $n$.  
		
		Last, we assume condition (3) and prove (2).  Pass to a subsequence so that 
		$$\lim_{n\to \infty} \left( \max_i \hat{h}_{E_{t_n}}(s_{i,n}) \right) = 0.$$
		Choose sequences of integers $a_{i,n}$ for $i = 1, \ldots,m$, not all 0, so that
		$$a_{1,n} (S_{1,t_n} - s_{1,n})+ \cdots + a_{m,n} (S_{m,t_n} - s_{m,n}) = O_{t_n}$$
		for all $n$.  Now, letting $M_n = \max_i a_{i,n}$, we can pass to a further subsequence so that
		$$\frac{a_{i,n}}{M_n} \to x_i \in \R$$
		as $n\to \infty$ for each $i$, with at least one $x_i$ nonzero.  
		This implies that 
		\begin{equation} \label{approximate 0}
			\hat{h}_{E_{t_n}}\left(\frac{1}{M_n}  (a_{1,n}S_{1,t_n} + \cdots + a_{m,n} S_{m,t_n}) \right) 
			= \hat{h}_{E_{t_n}} \left( \frac{1}{M_n}  (a_{1,n}s_{1,n}+ \cdots + a_{m,n} s_{m,n}) \right)   \to 0
		\end{equation}
		as $n\to \infty$.  Finally set
		$$X = x_1 S_1 + \ldots + x_m S_m.$$
		From Lemma \ref{one relation}, we know that the sequence $\{t_n\}$ has bounded height
		and by \cite[Theorem A]{Silverman:specialization} we get that the sequences $\{\hat{h}_{E_{t_n}}(S_{i,t_n})\}_n$ are bounded. 
		Therefore,  from the definition of $h_X$ in \eqref{associated height}, line \eqref{approximate 0} implies that $h_X (t_n) \to 0$.
	\end{proof}

	\subsection{Height 0}
	As we shall see, it follows from Theorem \ref{Lambda} that, although small sequences exist as in Proposition \ref{small points exist}, we don't always have sequences with height $0$:
	\begin{proposition}  \label{finite zeros}
		Fix nonzero $X \in E(k) \otimes \bR$.  There exist infinitely many $t\in B(\Kbar)$ for which $h_X(t) = 0$ if and only if there exists a real $c>0$ so that $c\, X$ is represented by an element of $E(k)$. \end{proposition}

	\begin{proof}
		Suppose first that $cX$ is represented by an element $P \in E(k)$ for some real $c>0$.  Then $h_X(t) = \frac{1}{c^2} h_P(t)$ at all $t$, so that $h_X(t) = 0$ whenever $P_t$ is torsion in $E_t$.  This holds at infinitely many points $t \in B(\Kbar)$.  (See, e.g., \cite[Proposition 6.2]{DM:variation}.)  
		
		For the converse, write $X = x_1 P_1 + \cdots + x_m P_m$ for independent $P_1, \ldots, P_m \in E(k)$ and $x_i\in \R$, and assume that $h_X(t)= 0$ for infinitely many $t$.  We can rewrite $X$ as 
		$$X = \alpha_1 Q_1 + \cdots + \alpha_s Q_s$$
		for $\alpha_1, \ldots, \alpha_s \in \R$ a basis for the span of $\{x_1, \ldots, x_m\}$ over $\bQ$ and $Q_1, \ldots, Q_s \in E(k)$.  For $s = 1$, we see that are we back in the setting where a multiple of $X$ is represented by an element of $E(k)$, so we may assume $s>1$.  But, for each $t$ where $h_X(t) =0$, we must have that $X_t = 0$ in $E_t(\Qbar)\otimes \R$.  By the choices of the $\alpha_i$, this means that each of the specializations $Q_{i,t}$ must be 0 in $E_t(\Qbar)\otimes \R$.  (Compare \cite[Lemma 1.1.1]{Moriwaki:Memoir}.)  In other words, the points $Q_1, \ldots, Q_s$ are simultaneously torsion at infinitely many $t$.  From Theorem \ref{Lambda}, combined with \eqref{Zhang for the sum}, this implies that each pair $Q_i$ and $Q_j$ is linearly related.  (Alternatively, here one could use the main results of \cite{Masser:Zannier, Masser:Zannier:2}.) Thus we infer that $X = c \, Q$ for some $c \in \R$ and $Q \in E(k)$.  
	\end{proof}

\subsection{Proof of Theorem \ref{geometry of relations}}  
From Theorem \ref{real point divisor theorem}, we know that $\Dbar_X$ is a continuous, adelic, semipositive, and normalized metrization on an ample $\R$-divisor.  Thus, Corollary \ref{equidistribution} applies to sequences with small height for $h_X$.  From Proposition \ref{small points exist}, we have $h_X(t_n) \to 0$ along any sequence $t_n$ for which $\sum_i r_{i,n} P_{i,t} = O_t$ with $r_{i,n} \in \bQ$ satisfying $r_{i,n} \to x_i$.  The formula for $\omega_{X,v}$ at each place follows from the definition of $\Dbar_X$ in \eqref{real point divisor}.  This completes the proof.  \qed

	
\bigskip
	
\section{The intersection number as a biquadratic form on $E(k)\otimes\bR$}
\label{biquadratic section}
	
	Let $\mcE \to B$ be a non-isotrivial elliptic surface defined over a number field $K$, and let $E$ be the corresponding elliptic curve over the field $k = \Kbar(B)$. Recall that, since $E(k)$ is finitely generated, we can pass to a finite extension of the number field $K$ to ensure that each section $P: B\to E$ is defined over $K$.  For each $P \in E(k)$, a metrized divisor $\Dbar_P$ is defined on the base curve $B$ by \eqref{Q-divisor}.  We extended this definition to elements $X \in E(k) \otimes \R$ with the definition \eqref{real point divisor}.  In this section, we study the basic properties of the Arakelov-Zhang intersection number 
	$$(X,Y) \mapsto \Dbar_X \cdot \Dbar_Y$$
defined by \eqref{pairing}, as a biquadratic form on the finite-dimensional vector space $E(k) \otimes \bR$.

Recall that the metrized $\R$-divisor $\Dbar_{\<X,Y\>} := \frac12 \left( \Dbar_{X+Y} - \Dbar_X - \Dbar_Y\right)$ was defined in \S\ref{bilinearity} for $X, Y \in E(k)\otimes \R$.  Our goal in this section is to prove:
	
\begin{proposition}\label{propertiesofpairing}
Fix $X, Y \in E(k)\otimes\bR$. The following hold:  
\begin{enumerate}
	\item  $\overline{D}_X\cdot \overline{D}_Y= \overline{D}_Y\cdot \overline{D}_X \; \geq \; 0$ and 
	\item $\Dbar_X \cdot \Dbar_{X+Y} = \Dbar_X \cdot \Dbar_Y$. 
\end{enumerate} 
Moreover, for each $X\in E(k)\otimes\bR$ the map $Y \mapsto \Dbar_X \cdot \Dbar_Y$ defines a positive semidefinite quadratic form on $E(k) \otimes \bR$, induced by the bilinear form $(Y,Z)\mapsto\Dbar_X\cdot\Dbar_{\<Y,Z\>}$.  
\end{proposition}

	We begin with a lemma:
	
\begin{lemma} \label{positivity}
We have $ \Dbar_X \cdot \Dbar_Y \geq 0$ for all $X,Y \in E(k)\otimes\bR$. 
\end{lemma}
	
\begin{proof}  From Theorem \ref{real point divisor theorem}, both $\Dbar_X$ and $\Dbar_Y$ are normalized, semipositive, continous adelic metrized divisors on $B$ over $K$, so the lemma follows immediately from Theorem \ref{isomorphic}.  Or we can see it as a consequence of Theorem \ref{two bundles} in the Appendix, because the height functions satisfy $h_X, h_Y \geq 0$ at all points of $B(\Kbar)$.
\end{proof}
	
The following lemma is a version of the Cauchy-Schwarz inequality. 
	
\begin{lemma} \label{bilinear}
For each $X \in E(k)\otimes \R$, the intersection
	$$(Y,Z) \mapsto \overline{D}_X\cdot \Dbar_{\<Y,Z\>}$$
is bilinear in $Y, Z \in   E(k)\otimes \R$.  Moreover, 
	$$(\overline{D}_X\cdot \Dbar_{\<Y,Z\>})^2\le (\overline{D}_X\cdot \overline{D}_Y)(\overline{D}_X\cdot \overline{D}_Z)$$
for all $X,Y,Z\in E(k)\otimes\bR$.
\end{lemma}
	
\begin{proof}
The bilinearity is an immediate consequence of the bilinearity demonstrated in \eqref{bilinear divisor} and the invariance of the intersection number under isomorphism.
		
Now fix $X, Y, Z \in E(k) \otimes \R$, and consider the function 
	$$f(x):=\overline{D}_X\cdot \overline{D}_{Y+xZ}.$$
By Lemma \ref{positivity} we have $f(x) \ge 0$ for all $x\in\bR$.  From definition \eqref{real point divisor}, we have 
	$$\Dbar_{Y + xZ} = (1-x) \Dbar_Y + x \Dbar_{Y+Z} + (x^2-x) \Dbar_Z.$$
Definition \eqref{inner product} then yields 
	$$f(x)=\overline{D}_X\cdot \overline{D}_{Y}+ 2x \overline{D}_X\cdot \Dbar_{\<Y,Z\>}+ x^2\overline{D}_X\cdot \overline{D}_{Z}\ge 0$$
for all $x\in\bR$.  Thus $f$ is a quadratic polynomial with non-positive discriminant.  The inequality follows.
\end{proof}
	
	We are now ready to prove the proposition.   
	
	\begin{proof}[Proof of Proposition \ref{propertiesofpairing}]
		Fix $X, Y, Z\in E(k)\otimes\R$. The symmetry in $(1)$ follows immediately from the symmetry of the intersection number, shown explicitly in \eqref{line bundle pairing} and extending to \eqref{pairing} by linearity.  The non-negativity is the content of Lemma \ref{positivity}.  
		
		
		For (2), we use Lemma \ref{bilinear} to compute that
		$$\left(\Dbar_X\cdot \Dbar_{\<X,Y\>}\right)^2 \leq \left(\Dbar_X\cdot\Dbar_X\right)\left(\Dbar_X\cdot\Dbar_Y \right) = 0$$
		because $\Dbar_X$ is normalized.  Therefore, 
		\begin{eqnarray*}
			0 &=& \Dbar_X\cdot \Dbar_{\<X,Y\>}  \\
			&=&  \frac12 \left( \Dbar_X\cdot \Dbar_{X+Y} - \Dbar_X\cdot \Dbar_X - \Dbar_X \cdot \Dbar_Y\right) \\
			&=& \frac12  \left(\Dbar_X\cdot \Dbar_{X+Y}  - \Dbar_X \cdot \Dbar_Y\right), 
		\end{eqnarray*}
		so that 
		$$\Dbar_X\cdot\Dbar_{X+Y} = \Dbar_X \cdot \Dbar_Y.$$
		
Finally, since $\Dbar_Y \iso \Dbar_{\<Y,Y\>}$ from \S\ref{bilinearity}, Lemma \ref{bilinear} then implies that $Y\mapsto \overline{D}_X\cdot \overline{D}_Y$ defines a positive semidefinite quadratic form as claimed.
	\end{proof}

\bigskip
\section{Equivalent formulations of Theorem \ref{Lambda}}
\label{equivalences}
	
Recall that $\mcE \to B$ denotes a non-isotrivial elliptic surface defined over a number field $K$, and let $E$ be the corresponding elliptic curve over the field $k = \Kbar(B)$.   We extend $K$ so that all sections of $\mcE\to B$ are defined over $K$.  {In this section, we prove the equivalence of Theorems \ref{Lambda} and \ref{scheme}.  We also provide in Theorem \ref{TFAE} a list of five additional, equivalent ways to express Theorem \ref{Lambda}.  One of these formulations, stated separately as Theorem \ref{rank drop}, is inspired by Zhang's conjecture in \cite[\S4]{Zhang:ICM}.

\subsection{Zhang's Conjecture for families of abelian varieties} 
In \cite{Zhang:ICM}, Zhang proposed the investigation of a function on the base curve $B$ that detects drops in rank of the specializations of a subgroup of $E(k)$:  given a finitely-generated subgroup $\Lambda$ of $E(k)$ of rank $m\geq 1$, if the quotient $\Lambda/\Lambda_{\mathrm{tors}}$ is generated by $S_1, \ldots, S_m \in E(k)$, let 
\begin{equation} \label{Zhang height function}
	h_{\Lambda}(t) \; := \; \det (\<S_i, S_j\>_t)_{i,j} \; \geq \; 0
\end{equation}
on $B(\Kbar)$, where defined, where $\<\cdot, \cdot\>_t$ is the N\'eron-Tate bilinear form on the specialization $\Lambda_t$ in the fiber $E_t$.    
	
We propose the following result as the analog of \cite[\S4 Conjecture]{Zhang:ICM} for elliptic surfaces; Zhang's conjecture was formulated for geometrically simple families of abelian varieties $\mathcal{A} \to B$ of relative dimension $>1$, and it does not hold as stated for elliptic surfaces \cite[\S4 Remark 3]{Zhang:ICM}.  
	
\begin{theorem} \label{rank drop}
Let $\mcE \to B$ be a non-isotrivial elliptic surface defined over a number field $K$, and let $E$ be the corresponding elliptic curve over the field $k = \Kbar(B)$.  Let $\Lambda \subset E(k)$ be a subgroup of rank $m\geq 2$, with the quotient $\Lambda/\Lambda_{\mathrm{tors}}$ generated by $S_1, \ldots, S_m \in E(k)$.  For each $i = 1, \ldots, m$, let $\Lambda_i \subset \Lambda$ be generated by $\{S_1, \ldots, S_m\}\setminus\{S_i\}$.  There is a constant $\epsilon=\epsilon(\Lambda)>0$ so that 
	$$\{t \in B(\Kbar):  h_{\Lambda_1}(t) + \cdots + h_{\Lambda_m}(t) \leq \eps\}$$ 
is finite. 
\end{theorem}
	
\noindent
We prove below that Theorem \ref{rank drop} is equivalent to Theorems \ref{Lambda} and \ref{scheme}.
	
\begin{remark}\label{hpositivedef}
Note that, for rank 1 groups $\Lambda$, the value $h_\Lambda(t)$ is the canonical height of the generating point $S_t$ in $E_t$.  In general, recall that the N\'eron-Tate height $\hat{h}_{E_t}$ on a smooth fiber over $t \in B(\Kbar)$ defines a positive definite quadratic form in $E_t(\Kbar)\otimes\mathbb{R}$; see e.g. \cite[Ch. VIII, Prop. 9.6]{Silverman:Elliptic}.  Thus,  $h_\Lambda$ will vanish at $t \in B(\Kbar)$ if and only if $\rank \Lambda_t < \rank \Lambda$.  The sum $h_{\Lambda_1}(t) + \cdots + h_{\Lambda_m}(t)$ will be zero if and only if the points $S_{1,t}, \ldots, S_{m,t}$ satisfy (at least) two independent linear relations over $\bZ$ in the fiber $E_t$. 
\end{remark}
	
\begin{remark}
The independence of the points $S_1, \ldots, S_m \in \Lambda$ in Theorem \ref{rank drop} is necessary for the finiteness statement to hold.  Indeed, suppose that $S_m$ is a linear combination of $S_1,\ldots,S_{m-1}$, and suppose that $\{t_n\}\subset B(\Kbar)$ is any infinite non-repeating sequence for which $h_{S_m}(t_n)\to 0$ (for example, we can take $t_n$ where $S_{m,t_n}$ is torsion; see e.g. \cite[Proposition 6.2]{DM:variation}).  It follows from Proposition \ref{regulatorvsreal} that $h_{\Lambda_1}(t_n) + \cdots + h_{\Lambda_m}(t_n)\to 0$.
\end{remark}

\subsection{Equivalences}
The remainder of this section is devoted to proving:
	
\begin{theorem} \label{TFAE}
Let $\mcE \to B$ be a non-isotrivial elliptic surface defined over a number field $K$, and let $E$ be the corresponding elliptic curve over the field $k = \Kbar(B)$.   Let $\Lambda$ be any subgroup of $E(k)$.  The following are equivalent:
\begin{enumerate}
	\item \label{condition Lambda} the conclusion of Theorem \ref{Lambda} holds for all $P,Q\in \Lambda$; 
	\item \label{condition scheme} the conclusion of Theorem \ref{scheme} holds for all sections $C$ of $\mcE^{\ell}$ defined by the graph $t\mapsto (Q_{1,t}, \ldots, Q_{\ell,t})$ for points $Q_1, \ldots, Q_{\ell} \in \Lambda$, for all $\ell\geq 2$; 
	\item \label{condition rank} the conclusion of Theorem \ref{rank drop} holds for this $\Lambda$; 
	\item \label{real nondegenerate}  the biquadratic form $(X,Y) \mapsto \Dbar_X\cdot\Dbar_Y$ on $\Lambda \otimes \R$ is non-degenerate, meaning that $\Dbar_X \cdot \Dbar_Y = 0$ if and only if $X$ and $Y$ are linearly dependent over $\R$; 
	\item \label{same height}  
for any pair $X,Y \in \Lambda \otimes\R$, if the heights satisfy  $h_X(t) = h_Y(t)$ for all $t\in B(\Kbar)$, then $X = \pm Y$;
	\item \label{orthogonal}  for any pair $X,Y \in \Lambda \otimes\R$, if the N\'eron-Tate inner product satisfies $\<X_t , Y_t\>_{E_t} = 0$ for all $t\in B(\Kbar)$ with $E_t$ smooth, then either $X$ or $Y$ is 0;
	\item \label{small sequence}  
for any pair $X,Y \in \Lambda \otimes\R$, if there exists an infinite (non-repeating) sequence of points $t_n \in B(\Kbar)$ for which 
	$$\lim_{n\to\infty}  h_X(t_n) + h_Y(t_n) = 0,$$
then $X$ and $Y$ are linearly dependent over $\bR$. 
\end{enumerate}
\end{theorem}

	For the proof, we rely on the work carried out in Sections \ref{R-divisors} -- \ref{biquadratic section}.  Specifically, for each $X \in E(k)\otimes \R$, we can express $X$ as a finite $\R$-linear combination of elements $P_1, \ldots, P_m \in E(k)$.  We appeal to Theorem \ref{real point divisor theorem} to know that $\Dbar_X$ is a well-defined, semipositive, normalized, continuous adelic metrization on $B$, defined over the number field $K$.  Further, $(X,Y) \mapsto \Dbar_X \cdot \Dbar_Y$ is a well-defined semipositive biquadratic form on $E(k)\otimes \R$ by Proposition \ref{propertiesofpairing}.

\subsection{Intersection number 0}
Towards a proof of Theorem \ref{TFAE}, we first examine the consequences of the existence} of a pair $X, Y\in E(k)\otimes \R$ for which $\Dbar_X \cdot \Dbar_Y = 0$.  
	
Recall that $\<\cdot, \cdot\>_t$ denotes the N\'eron-Tate bilinear form on the fiber $E_t(\Kbar)\otimes\R$, and $\<\cdot, \cdot\>_E$ denotes the corresponding form on $E(k)\otimes \R$.  
	
\begin{proposition}\label{nondegeneracylevel1} 
Fix nonzero $X,Y\in E(k)\otimes \R$, and assume that $\Dbar_X \cdot \Dbar_Y = 0$.  Then for all $t\in B(\overline{K})$ for which the fiber $E_t$ is smooth, we have 
	$$	h_X(t) \; =\; \frac{\hat{h}_E(X)}{\hat{h}_E(Y)} \;h_Y(t) \quad\mbox{ and } \quad
		\<X_t,Y_t\>_t \; =\; \frac{\<X,Y\>_E}{\hat{h}_E(Y)}\; h_Y(t).$$
Moreover, we have $\Dbar_{X'} \cdot \Dbar_{Y'} = 0$ for all $X', Y' \in \mathrm{Span}_\R(\{X,Y\})$.
\end{proposition}
	
\begin{proof}	
Assume that $\overline{D}_X\cdot \overline{D}_Y=0$.  From Theorem \ref{real point divisor theorem}, each of $\Dbar_X$ and $\Dbar_Y$ is a continuous, normalized, semipositive adelic metrization on an $\R$-divisor on $B$.  The degree of $D_X$ (respectively $D_Y$) is $\hat{h}_E(X)$ (respectively, $\hat{h}_E(Y)$).  The relation between the heights $h_X$ and $h_Y$ follows immediately from Theorem \ref{isomorphic+}.
		
Using now part (2) of Proposition \ref{propertiesofpairing} our assumption that $\overline{D}_X\cdot \overline{D}_Y=0$ implies that $\Dbar_{X+Y}\cdot \Dbar_Y= \overline{D}_{X+Y}\cdot \overline{D}_X = 0$, and so
\begin{eqnarray*}
\Dbar_{xX + yY} \cdot \Dbar_{aX + bY} &=&  \left( (x^2-xy)\Dbar_X + xy \Dbar_{X+Y} + (y^2-xy) \Dbar_Y \right) \\
		&&  \qquad \cdot \left( (a^2-ab)\Dbar_X + ab \Dbar_{X+Y} + (b^2-ab) \Dbar_Y \right) \\
		&=& 0
\end{eqnarray*}
for all $x, y, a, b, \in \R$.  In particular, we have 
	$$	\hat{h}_{E_t}(X_t+Y_t) =\frac{\hat{h}_E(X+Y)}{\hat{h}_E(Y)}\, \hat{h}_{E_t}(Y_t).$$
so that
	$$\hat{h}_{E_t}(X_t+Y_t)=\hat{h}_{E_t}(X_t)+\<X_t,Q_t\>_t +\hat{h}_{E_t}(Y_t)$$ 
implies
	$$\<X_t,Y_t\>_t \; =\; \frac{\<X,Y\>_E}{\hat{h}_E(Y)}\; h_Y(t)$$
for all $t \in B(\Kbar)$ for which $E_t$ is smooth.
\end{proof}
	
The following proposition extends the observations of Proposition \ref{regulatorvsreal} to two independent relations.
	
\begin{proposition}  \label{almost 2 relations}
Let $\Lambda$ be a subgroup of $E(k)$ generated by independent, non-torsion elements $P_1, \ldots, P_m$, with $m\geq 2$.  The following are equivalent:  
\begin{enumerate}
	\item  there exist an infinite, non-repeating sequence $t_n \in B(\Kbar)$ and points $p_{i,n} \in E_{t_n}(\Kbar)$ for $i = 1, \ldots m$, for which $\hat{h}_{E_{t_n}}(p_{i,n}) \to 0$ as $n\to \infty$, and the points 
		$$P_{1, t_n} - p_{1,n}, \ldots, P_{m,t_n} - p_{m,n}$$
	satisfy two independent linear relations on $E_{t_n}$;  
	\item  there exist independent $X, Y \in \Lambda \otimes \R$ for which 
		$$\Dbar_X \cdot \Dbar_Y = 0.$$
\end{enumerate}
\end{proposition}

	\begin{proof}
		Assume first that $\Dbar_X \cdot \Dbar_Y = 0$.  Write $X = x_1P_1 + \ldots + x_mP_m$ and $Y = y_1P_1 + \cdots + y_mP_m$ for linearly independent coefficient vectors $\vec{x}, \vec{y} \in \R^m$.   From Proposition \ref{nondegeneracylevel1}, we can replace $X$ and $Y$ by linear combinations of $X$ and $Y$ (and relabel the points $P_i$ if needed) and so assume that $x_1 = 1 = y_m$ and $x_m = y_1 = 0$.  From Theorem \ref{real point divisor theorem}, we know that $\Dbar_X$ and $\Dbar_Y$ are normalized, semipositive, continuous adelic metrizations.  By Proposition \ref{small points exist}, we know that $e_1(\Dbar_X) = e_1(\Dbar_Y) = 0$.  Theorem \ref{isomorphic+} then implies that there is an infinite non-repeating sequence $\{t_n\}\subset B(\Kbar)$ so that 
		\begin{align}\label{end1}
			h_X(t_n) + h_Y(t_n) \to 0
		\end{align}
		as $n\to \infty$.  
		We now apply Proposition \ref{regulatorvsreal} to each of $h_X$ and $h_Y$ to show that small perturbations of the specializations $P_{i,t_n}$ must satisfy two independent relations in the fibers $E_{t_n}(\Kbar)$.  More precisely, we choose integers $a_{i,n}, b_{i,n}$ for each $n\geq 1$ and each $i =2, \ldots, m-1$, so that 
		$$\frac{a_{i,n}}{n} \to x_i \quad\mbox{ and } \quad \frac{b_{i,n}}{n} \to y_i$$
		as $n\to \infty$.  As in the proof of Proposition \ref{regulatorvsreal} (2) $\implies$ (3), we choose $p_n \in E_{t_n}(\Kbar)$ so that 
		$$n \, p_n = a_{2,n} P_2 + \cdots + a_{m-1,n}P_{m-1}.$$
		Set $p_{1,n} = P_{1,t_n} + p_n \in E_{t_n}(\Kbar)$.  Then 
		$$\hat{h}_{E_{t_n}}(p_{1,n}) = \hat{h}_{E_{t_n}}\left( P_{1,t_n} + \frac{1}{n} (a_{2,n} P_2 + \cdots + a_{m-1,n}P_{m-1}) \right) \to 0,$$
		and $\{P_{1,t_n} - p_{1,t_n}, P_{2,t_n}, \ldots, P_{m-1,t_n}\}$ satisfy a linear relation.  On the other hand, we can repeat the same argument with $Y$ and find point $q_n \in E_{t_n}(\Kbar)$ so that 
		$$n \, q_n = b_{2,n} P_2 + \cdots + b_{m-1,n}P_{m-1}$$
		and set $p_{m,n} = P_{m,t_n} + q_n$.  Then 
		$$\hat{h}_{E_{t_n}}(p_{m,n}) = \hat{h}_{E_{t_n}}\left(  \frac{1}{n} (b_{2,n} P_2 + \cdots + b_{m-1,n}P_{m-1}) + P_{m,t_n} \right) \to 0,$$
		and $\{P_{2,t_n}, \ldots, P_{m-1,t_n}, P_{m,t_n} - p_{m,t_n}\}$ satisfy a linear relation.  It follows that the points
		$$\{P_{1,t_n} - p_{1,t_n}, P_{2,t_n}, \ldots, P_{m-1,t_n}, P_{m,t_n} - p_{m,t_n}\}$$
		satisfy two independent linear relations in $E_{t_n}(\Kbar)$ for all $n$.  
		
		For the converse direction, we assume there are an infinite, non-repeating sequence $t_n \in B(\Kbar)$ and points $p_{i,n} \in E_{t_n}(\Kbar)$ for $i = 1, \ldots m$, for which $\hat{h}_{E_{t_n}}(p_{i,n}) \to 0$ as $n\to \infty$, and such that 
		$$\{P_{1, t_n} - p_{1,n}, \ldots, P_{m,t_n} - p_{m,n}\}$$
		satisfy two independent linear relations on $E_{t_n}$.   From Lemma \ref{one relation}, we know that the sequence $\{t_n\}$ must have bounded height.  Choose integers $a_{i,n}, b_{i,n}$ for $n\geq 1$ and $i = 1, \ldots, m$ so that the independent relations are expressed as 
		$$a_{1,n}(P_{1, t_n} - p_{1,n}) + \cdots + a_{m,n}( P_{m,t_n} - p_{m,n}) = O_{t_n}$$
		and 
		$$b_{1,n}(P_{1, t_n} - p_{1,n}) + \cdots + b_{m,n}( P_{m,t_n} - p_{m,n}) = O_{t_n}$$
		Relabeling the points if necessary, we can rewrite the expressions as 
		$$(P_{1, t_n} - p_{1,n}) + r_{2,n} (P_{2, t_n} - p_{2,n}) + \cdots + r_{m,n}( P_{m,t_n} - p_{m,n}) = O_{t_n}$$
		and 
		$$r_{1,n}'(P_{1, t_n} - p_{1,n}) + \cdots + r_{m-1,n}'( P_{m-1,t_n} - p_{m-1,n}) + ( P_{m,t_n} - p_{m,n}) = O_{t_n}$$
		for bounded sequences of rational numbers $r_{2,n}, \ldots, r_{m,n}$ and $r_{1,n}', \ldots, r_{m-1,n}'$. Passing to a subsequence we may assume that 
		$$r_{i,n} \to x_i \in \R \quad \mbox{ and } \quad r_{i,n}' \to y_i \in \R$$
		for each $i$.  Then, recalling that $\{t_n\}$ has bounded height and that the perturbations $p_{i,n}$ have heights tending to 0 and using \cite[Theorem A]{Silverman:specialization} to infer that $\{\hat{h}_{E_{t_n}}( P_{i,t_n})\}_n$ are bounded for each $i$, we conclude that 
		$$h_X(t_n)  \to 0 \quad \mbox{ and } \quad h_Y(t_n) \to 0$$
		along this subsequence, for $X = P_1 + x_2P_2 + \ldots + x_mP_m$ and $Y = y_1P_1 + \cdots + y_{m-1}P_{m-1} + P_m$.  From Theorem \ref{isomorphic+}, we have that $\Dbar_X \cdot \Dbar_Y = 0$.
	\end{proof}

\subsection{Proof of Theorem \ref{TFAE}}  Throughout this proof, we fix a finitely-generated subgroup $\Lambda \subset E(k)$.  Assume it is of rank $m\ge 1$ with $\Lambda/\Lambda_{\mathrm{tors}}$ generated by $P_1, \ldots, P_m \in E(k)$. 
	
\bigskip
\noindent\fbox{$\eqref{condition Lambda}\iff \eqref{real nondegenerate}$}
Recall that the N\'eron-Tate height $\hat{h}_E$ on $\Lambda$ extends to a positive definite quadratic form on $\Lambda\otimes \R$.  It follows (by Cauchy-Schwarz) that the N\'eron-Tate regulator 
	$$R_E(X,Y) := \hat{h}_E(X)\hat{h}_E(Y) - \<X,Y\>_E^2 \;\geq \; 0$$ 
extends to a biquadratic form on $\Lambda\otimes\R$ satisfying $R_E(X,Y) = 0$ if and only if $X$ and $Y$ are linearly dependent over $\R$.  As 
	$$F(X,Y) := \Dbar_X \cdot \Dbar_Y$$
is also biquadratic on $\Lambda\otimes\R$ from Proposition \ref{propertiesofpairing}, and it satisfies $F(X,X) = 0$ for all $X \in E(k) \otimes \R$, the upper bound on $\Dbar_X \cdot \Dbar_Y$ in Theorem \ref{Lambda} follows.  Condition \eqref{condition Lambda} is then equivalent to the statement that $F(X,Y)=0$ if and only if $X$ and $Y$ are linearly dependent over $\R$.  
	
In details, if we assume \eqref{condition Lambda}, and if the pair $X = \sum_{i=1}^m x_i P_i$ and $Y = \sum_{i=1}^m y_i P_i$ with $P_i \in \Lambda$ satisfy $\Dbar_X\cdot \Dbar_Y = 0$, then we can approximate by rational combinations $P_n = \frac1n \sum a_{i,n} P_i \to X$ and $Q_n = \frac1n \sum b_{i,n} P_i \to Y$ with integers $a_{i,n}, b_{i,n}$, and compute that 
	$$\Dbar_{P_n} \cdot \Dbar_{Q_n} = \frac{1}{n^4} \,\Dbar_{\sum a_{i,n}P_i} \cdot \Dbar_{\sum b_{i,n}P_i} \; \geq \; \frac{c}{n^4} \, R_E\left(\sum a_{i,n}P_i, \sum b_{i,n}P_i \right) =  c \, R_E(P_n, Q_n).$$
Letting $n\to \infty$ shows that $R_E(X,Y)=0$, implying that $X,Y$ are linearly dependent over $\R$.  Now assume \eqref{real nondegenerate}, so that $F(\cdot,\cdot)$ is nondegenerate on the finite-dimensional $V = \Lambda \otimes \R$.  Using the inner product $\<\cdot, \cdot\>_E$ on $V$ and associated norm $\|\cdot \| = \hat{h}_E(\cdot)^{1/2}$, we have (by continuity and compactness) uniform positive upper and lower bounds on $\Dbar_X \cdot \Dbar_Y$ over all pairs $X,Y\in V$ satisfying $\<X, Y\>_E = 0$ and $\|X\| = \|Y\| = 1$.  On the other hand, $R_E(X,Y) = 1$ for all such pairs, and so there is a positive constant $c = c(V)$ so that 
\begin{equation} \label{comparable}
	c \, R_E(X,Y) \leq \Dbar_X \cdot \Dbar_Y \leq c^{-1} R_E(X,Y)
\end{equation}
all pairs $X,Y\in V$ satisfying $\<X, Y\>_E = 0$ and $\|X\| = \|Y\| = 1$.  By scaling the points, this extends to orthogonal pairs of any norm.  For an arbitrary pair $X,Y \in V$, we write $Y = Y' + xX$ with $\<Y',X\>_E=0$ and $x\in \R$, and observe that $\Dbar_X\cdot \Dbar_Y = \Dbar_X \cdot \Dbar_{Y'}$ from Proposition \ref{propertiesofpairing}.  We also have $R_E(X, Y'+xX) = R_E(X,Y')$ and so \eqref{comparable} holds for all pairs $X$ and $Y$ in $V$.

\bigskip
\noindent\fbox{$\eqref{real nondegenerate}\iff \eqref{small sequence}$}
Fix any pair $X, Y \in \Lambda\otimes \R$, and express $X$ and $Y$ as $\R$-linear combinations of elements $P_1, \ldots, P_m \in \Lambda$.  Theorem \ref{real point divisor theorem} shows that $\Dbar_X$ and $\Dbar_Y$ are normalized, continuous, semipositive adelic metrizations on $\R$-divisors, and Proposition \ref{small points exist} shows that each has essential minimum equal to $0$.  Theorem \ref{isomorphic+} then implies that $\Dbar_X \cdot \Dbar_Y=0$ if and only if the heights $h_X$ and $h_Y$ have a common small sequence in $B(\Kbar)$.
	
\bigskip
\noindent\fbox{$\eqref{condition rank} \iff \eqref{small sequence}$}
Assume that \eqref{small sequence} holds.  We aim to prove the conclusion of Theorem \ref{rank drop} for this $\Lambda$.  Suppose that there is an infinite, non-repeating sequence $t_n\in B(\overline{K})$ with $h_{\Lambda_i}(t_n)\to 0$ for all $i=1,\ldots,m$.  From Lemma \ref{regulatorvsreal}, we may choose point $X_1 \in \Lambda_1$ so that $\liminf_{n \to\infty} h_{X_1}(t_n) = 0$.  Pass to a subsequence so that $\lim_{n\to\infty} h_{X_1}(t_n) = 0$.  For each $i = 2, \ldots, m$, we successively apply Lemma \ref{regulatorvsreal} to find $X_i \in \Lambda_i$ for which $\liminf_{n \to\infty} h_{X_i}(t_n) = 0$ and then pass to a further subsequence so that $\lim_{n\to\infty} h_{X_i}(t_n) = 0$.  In this way, we have an infinite, non-repeating sequence of points $t_n \in B(\Kbar)$ so that $\lim_{n\to\infty} h_{X_i}(t_n) = 0$ for all $i$.  However, as $\bigcap_{i=1}^{m}\Lambda_i=\{0\}$, at least two of the $X_i$ must be independent.  This contradicts \eqref{small sequence}. 
	
Assume now that \eqref{condition rank} holds.  Fix a pair of independent nonzero points $X, Y \in \Lambda\otimes \R$, and suppose that there is an infinite nonrepeating sequence $t_n\in B(\Kbar)$ so that 
	$$h_X(t_n)+h_Y(t_n)\to 0.$$  
Then $h_X(t_n)\to 0$ and $h_Y(t_n)\to 0$. 
We write 
\begin{align*}
	\begin{split}
	X&=a_1P_1+\cdots+a_mP_m\\
	Y&=b_1P_1+\cdots+b_mP_m,
\end{split}
\end{align*}
with $a_i,b_j\in \bR$ and independent $P_i\in \Lambda$. 
We want to show that 
	\begin{align}\label{alllat}
		\displaystyle\liminf_{n\to\infty}h_{\Lambda_i}(t_n)=0
	\end{align}
for all $i=1,\ldots,m$, contradicting \eqref{condition rank}. Fix $i\in\{1,\ldots,m\}$.  If $b_i=0$, then $Y\in\Lambda_i\otimes\bR$ and \eqref{alllat} follows from Lemma \ref{regulatorvsreal}.  If on the other hand $b_i\neq 0$, then $X-\frac{a_i}{b_i}Y\in\Lambda_i\otimes\bR$ and by the parallelogram law we also have
	$$h_{X-\frac{a_i}{b_i}Y}(t_n)\to 0.$$
As before, equation \eqref{alllat} follows by Lemma \ref{regulatorvsreal}.

\bigskip
\noindent\fbox{$\eqref{small sequence} \implies \eqref{orthogonal}$}
Assume that \eqref{small sequence} holds. Fix $X,Y\in \Lambda\otimes\bR$, and suppose that $\<X_t,Y_t\>_t=0$ for all $t\in B(\Kbar)$ for which the fiber $E_t$ is smooth.  By Proposition \ref{small points exist} there is an infinite sequence $t_n\in B(\Kbar)$ with $h_{X-Y}(t_n)\to 0$.  Since $\<X_{t_n},Y_{t_n}\>_{t_n}=0$ we have 
	\begin{eqnarray*}
		h_X(t_n)+h_Y(t_n)&=& h_X(t_n)-2\<X_{t_n},Y_{t_n}\>_{t_n}+h_Y(t_n)\\
		&=&  h_{X-Y}(t_n) \; \to \; 0.  
	\end{eqnarray*}
Thus by \eqref{small sequence} we get that  either $X$ or $Y$ is $0$ or there are non-zero $a,b\in\bR$ such that $aX=bY$. In the latter case, our assumption that $\<X_t,Y_t\>_t=0$ for all $t$ implies that both $X$ and $Y$ are $0$. The assertion follows.

\bigskip
\noindent\fbox{$\eqref{orthogonal} \implies \eqref{same height}$}
Assume that \eqref{orthogonal} holds.  Fix $X,Y\in \Lambda\otimes\bR$, and suppose that $h_X(t)=h_Y(t)$ for all $t\in B(\Kbar)$.  If $X=0$ or $Y=0$ then our assumption that $h_X(t)=h_Y(t)$ for all $t$ implies that $X=Y=0$ in $\Lambda\otimes\bR$. Thus we may assume that both $X$ and $Y$ are non-zero.  Since $h_X(t)=h_Y(t)$ for all $t$, Silverman's specialization theorem \cite[Theorem B]{Silverman:specialization} implies that $\hat{h}_E(X)=\hat{h}_E(Y)$.   From Theorem \ref{isomorphic+} we know that $\Dbar_X \cdot \Dbar_Y = 0$, and therefore, from Proposition \ref{nondegeneracylevel1}, we have 
	$$\<X_t,Y_t\>_t=\frac{\<X,Y\>_E}{\hat{h}_E(Y)} \, \hat{h}_{E_t}(Y_t),$$ 
or equivalently, 
	$$	\left\<X_t- \frac{\<X,Y\>_E}{\hat{h}_E(Y)} \, Y_t,\; Y_t\right\>_t=0,$$
for all $t$. By our assumption \eqref{orthogonal} and since $Y\neq 0$, we have 
	$$X =  \frac{\<X,Y\>_E}{\hat{h}_E(Y)} \, Y.$$
Recalling that $\hat{h}_E(X)=\hat{h}_E(Y)$, we get $X=\pm Y$ in $\Lambda \otimes \R$, as claimed.

\bigskip
\noindent\fbox{$\eqref{same height}\implies  \eqref{small sequence}$} 
Suppose there exist nonzero $X,Y\in \Lambda\otimes\bR$ and an infinite, non-repeating sequence $t_n\in B(\Kbar)$ for which $h_X(t_n)+h_Y(t_n)\to 0$.  By Theorem \ref{real point divisor theorem} we know that both $h_X$ and $h_Y$ are induced by normalized semipositive adelic metrizations on ample divisors $D_X$ and $D_Y$ on $B$, of degrees $\hat{h}_E(X)$ and $\hat{h}_E(Y)$, respectively.  We may thus apply Theorem \ref{isomorphic+} to get 
	$$h_X(t)=\frac{\hat{h}_E(X)}{\hat{h}_E(Y)} \, h_Y(t)=h_{xY}(t)$$
for all $t$ in $B(\Kbar)$, where $x=\sqrt{\frac{\hat{h}_E(X)}{\hat{h}_E(Y)}}$. Our assumption \eqref{same height} then yields that $X=\pm xY$ as claimed.

\bigskip
\noindent\fbox{$\eqref{condition scheme} \iff \eqref{real nondegenerate}$}  Fix any collection of points $Q_1, \ldots, Q_{\ell}$ in $\Lambda$, and let $C$ be the irreducible curve in $\mathcal{E}^\ell$ defined by a section $(Q_1, \ldots, Q_{\ell})$ over $B$.  To say that $C$ is not contained in flat subgroup scheme of positive codimension means that the points $Q_1, \ldots, Q_{\ell}$ are linearly independent.  To say that the curve $C$ in $\mcE^\ell$ defined by $(Q_1, \ldots, Q_{\ell})$ intersects the tube $T(\mathcal{E}^{m,\{2\}},\epsilon)$ infinitely often for every $\epsilon>0$ means that there is an infinite non-repeating sequence of points $t_n \in B(\Kbar)$ and small points $q_{i,n} \in E_{t_n}(\Kbar)$ for each $n$ so that the points $\{Q_{1,t_n}- q_{1,n}, \ldots, Q_{\ell,t_n}- q_{\ell,n}\}$ satisfy two linear relations in $E_{t_n}$.  Therefore the equivalence of \eqref{condition scheme} and \eqref{real nondegenerate} is the statement of Proposition \ref{almost 2 relations}.  
	
This completes the proof of the theorem.

\bigskip
\section{Equality of measures}
\label{measure section}
	
In this section we prove Theorem \ref{measures}, which is needed for our proof of Theorems \ref{Lambda} and \ref{scheme}.  We begin by introducing a complex-geometric perspective on the elements $X$ of the real vector space $E(k)\otimes \R$.  These points do not necessarily exist as algebraic curves in the elliptic surface $\mcE \to B$ but can be viewed as inducing foliations.
	
\subsection{Real points as holomorphic curves}  \label{foliation section}
Given a non-isotrivial elliptic surface $\mcE \to B$ defined over the number field $K$, we fix an embedding $K \hookrightarrow \C$, and let $S \subset B$ be a finitely-punctured Riemann surface so that all fibers $E_t(\C)$ are smooth for $t \in S(\C)$.  Write $\mcE_S$ for the open subset of $\mcE$ over $S$.  Recall that each rational point $P \in E(k)$ determines a holomorphic section of $\mcE\to B$ defined by $t \mapsto P_t \in E_t(\C)$ for $t \in S(\C)$. 
	
The Betti coordinates of $P \in E(k)$ are defined as follows.  Passing to the universal cover $\pi: \Stilde \to S$, there is a holomorphic period function 
	$$\tau:  \Stilde \to \Hyp$$
taking values in the upper half plane, so that the fibers of $\mcE_S$ satisfy
	$$E_{\pi(s)}(\C) \iso \C/ (\bZ\oplus \bZ\tau(s))$$
for all $s \in \Stilde$.  Passing to the universal cover of $E_{\pi(s)}(\C)$ for each fiber, we obtain a holomorphic line bundle over $\Stilde$, trivialized by sending generator 1 of the lattice to $1 \in \C$.  For each $P \in E(k)$, the corresponding section of $\mcE\to B$ lifts to a holomorphic function 
	$$\xi_P:  \Stilde \to \C.$$
The Betti map of $P$ is the real-analytic map $\beta_P: \Stilde \to \R^2$ given by 
	$$\beta_P(s) = (x(s), y(s)) \quad\mbox{ such that }\quad  \xi_P(s) = x(s) + y(s) \, \tau(s).$$   
The coordinates $x$ and $y$ themselves depend on the choices of $\tau$ and $\xi_P$, but as proved in \cite{CDMZ}, we have 
\begin{equation} \label{Betti measure}
	\omega_P = dx \wedge dy,
\end{equation}
independent of the choices, for the curvature distribution of $\Dbar_P$ at an archimedean place of $K$.  
	
Given $P \in E(k)$ and a fixed choice of $\xi_P$, and given a nonzero integer $n$, the holomorphic function 
		$$\xi := \frac{1}{n} \, \xi_P$$
will represent a point $Q\in E(\kbar)$ satisfying $n \, Q = P$.  It descends to a holomorphic curve in $\mcE_S$ that is not necessarily a section over $S$.  Translating $\xi$ by elements of $\frac{1}{n}(\bZ\oplus \bZ\tau)$, we find all curves corresponding to solutions $Q$ of $n \, Q = P$.  More generally, we find that every element of $E(k)\otimes \R$ can be represented by a family of holomorphic curves in $\mcE_S$, as follows: 
	
\begin{proposition} \label{foliation}
Fix a choice of period function $\tau: \Stilde \to \Hyp$, and suppose that $P_1, \ldots, P_m \in E(k)$ provide a basis for $E(k) \otimes \R$.  Then there exist Betti coordinates for each $X= \sum_i x_i P_i  \in E(k) \otimes \R$, given by
		$$\beta_X(s) = (x_X(s), y_X(s)) = \sum_i x_i \,\beta_{P_i}(s)  + (a,b) $$
for $s \in \Stilde$, for any choices of Betti coordinates $\beta_{P_i}$ for the points $P_i$ and any constant $(a,b) \in \R^2$, so that the curvature distribution for $\Dbar_X$ at an archimedean place of $K$ satisfies
		$$\omega_X = dx_X \wedge dy_X$$
on $S$. 
\end{proposition}
	
Recall here that the curvature distribution for $\Dbar_X$ (at the given archimedean place) was defined in  \eqref{real point divisor} by
	$$\omega_{X} \; =  \; 
	\sum_i \left(x_i^2 - \sum_{j\not=i} x_ix_j \right)  \omega_{P_i} + \sum_{i < j} x_ix_j \; \omega_{P_i + P_j}, $$
	with $P_i\in E(k)$.

\begin{remark}  
Given $X \in E(k)\otimes \R$, the family of holomorphic functions $\xi_X(s) := x_X(s) + y_X(s)\, \tau(s)$ of Proposition \ref{foliation} projects to a family of holomorphic curves in the complex surface $\mcE_S$.  For torsion points of $E(k)$ representing the 0 of $E(k)\otimes \R$, the holomorphic curves given by Proposition \ref{foliation} are precisely the leaves of the Betti foliation, because we allow for arbitrary translation of $\beta_X$ in $\R^2$.  By definition, the leaves of the Betti foliation have constant Betti coordinates; see, e.g., \cite{ACZ:Betti}, \cite{CDMZ}, and \cite{UU2} for more information.  For each nonzero $X \in E(k)\otimes\R$, there is a corresponding foliation of $\mcE_S$.  When an element $X$ is represented by $P \in E(k)$, the foliation is simply the corresponding Betti foliation for the elliptic surface with $P$ chosen as the zero section.
\end{remark}
	
\begin{proof}[Proof of Proposition \ref{foliation}]
Let $A_n$ be a sequence in $E(k)$ such that $Q_n := \frac{1}{n} A_n$ converges to $X$ in $E(k) \otimes \R$ as $n \to \infty$.  We can select the holomorphic lifts $\xi_{A_n}: \Stilde \to \C$ and choices of $\xi_{Q_n}$ so that the sequence of holomorphic functions $\xi_{Q_n}$ converges, locally uniformly in $\Stilde$.  This defines a limit holomorphic function $\xi_X$.  In terms of a basis $P_1, \ldots, P_m$ of $E(k)$, we can assume that $Q_n = \frac{1}{n} \left( a_{n,1} P_1 + \cdots + a_{n,m} P_m\right)$  for integers $a_{n,i}$, with $a_{n,i}/n \to x_i \in \R$ as $n\to \infty$.  We see that $\xi_X  - \sum_i x_i \xi_{P_i}$ must be an element of $\R\oplus \R\tau$.  Making other choices for $\xi_{A_n}$ and $\xi_{Q_n}$, we can obtain all possible translates of $\xi_X$ by elements of $\R\oplus \R\tau$; in other words, we can define Betti coordinates for $X$, up to translation by elements of $\R^2$.  Fix a choice of $\beta_X = (x_X, y_X)$ and consider the measure $\nu_X = dx_X \wedge dy_X$.  This measure is clearly independent of the choices.  Furthermore, it is the weak limit of measures $\omega_{Q_n}$ on $S$, by the formula \eqref{Betti measure} for $\omega_{Q_n}$ and local uniform convergence of $\xi_{Q_n}$ to $\xi_X$.  We already know that $\omega_{Q_n} \to \omega_X$ for the curvature distributions (at a fixed archimedean place), from the definitions given in \S\ref{R divisor subsection}.  It follows that $\nu_X = \omega_X$. 
\end{proof}

\subsection{Proof of Theorem \ref{measures}}
Fix $X_1, X_2 \in E(k) \otimes \R$, and let $\Dbar_1$ and $\Dbar_2$ be the associated metrized $\R$-divisors on $B$, defined over the number field $K$.  Fix an archimedean place of $K$, and let $\omega_1$ and $\omega_2$ be the curvature measures on $B(\C)$ at this place.  We assume that $\omega_1 = \omega_2$.   As in \S\ref{foliation section}, we fix a period function $\tau: \Stilde \to \Hyp$.   From Proposition \ref{foliation}, there exist holomorphic functions $\xi_i = x_i + y_i \tau$, $i = 1,2$, representing the points $X_1$ and $X_2$, so that
\begin{equation} \label{m1}
	dx_1 \wedge dy_1 = dx_2 \wedge dy_2
\end{equation}
on $\Stilde$.  
	
We break the proof into two steps.  In the first, we exploit the holomorphic-antiholomorphic trick of \cite[\S5]{ACZ:Betti}, applied to a relation between holomorphic functions $\xi_1$, $\xi_2$, $\tau$ (and their derivatives) and the anti-holomorphic functions $\bar\xi_1$, $\bar\xi_2$, and $\bar\tau$ (and their derivatives) coming from \eqref{m1}; the result is a relation on the holomorphic input alone.  In the second step, we apply the transcendence result of \cite[Th\'eor\`eme 5]{Bertrand} to this relation and deduce that the points $X_1$ and $X_2$ must be linearly related in $E(k) \otimes \R$.  
	
\medskip
\noindent{\bf Step 1:  Holomorphic-Antiholomorphic.}  We are grateful to Lars K\"uhne for teaching us this step.
	
Note that 
	$$d\xi_i = dx_i + y_i \, d\tau + \tau \, dy_i$$
so that 
	$$(d\xi_i - y_i d\tau)\wedge (d\bar\xi_i - y_i d\bar\tau) = (\bar\tau - \tau) \, dx_i \wedge dy_i.$$
Writing 
	$$y_i = \frac{\xi_i - \bar\xi_i}{\tau - \bar\tau}$$
we obtain a relation from \eqref{m1} expressed as 
\begin{multline*}
( (\tau-\bar\tau) d\xi_1 - (\xi_1-\bar\xi_1) d\tau) \wedge ( (\tau-\bar\tau) d\bar\xi_1 - (\xi_1-\bar\xi_1) d\bar\tau)  \\  = ( (\tau-\bar\tau) d\xi_2 - (\xi_2-\bar\xi_2) d\tau) \wedge ( (\tau-\bar\tau) d\bar\xi_2 - (\xi_2-\bar\xi_2) d\bar\tau)
\end{multline*}
as forms on $\tilde{S}$.
	
Working in coordinates in the simply connected $\Stilde$, this gives 
\begin{multline} \label{explicit relation}
\left( \xi_1' \overline{\xi_1'} - \xi_2' \overline{\xi_2'} \right) (\tau-\bar\tau)^2 - \left( (\xi_1 - \overline{\xi_1}) \xi_1' - (\xi_2 - \overline{\xi_2}) \xi_2'\right) (\tau - \bar\tau) \overline{\tau'}   \\  -    \left( (\xi_1 - \overline{\xi_1}) \overline{\xi_1'} - (\xi_2 - \overline{\xi_2}) \overline{\xi_2'} \right) (\tau - \bar\tau) \tau' + \left( (\xi_1 - \overline{\xi_1})^2 - (\xi_2 - \overline{\xi_2})^2 \right) \tau' \overline{\tau'}  \;\; = \;\;  0
\end{multline}
as functions on $\tilde{S}$.  Equation \eqref{explicit relation} can be expressed as
	$$\sum_{j=1}^N f_j(z) g_j(z) \equiv 0$$
for holomorphic functions $f_j \in \bZ[\xi_1, \xi_2, \xi_1', \xi_2', \tau, \tau']$ and antiholomorphic functions $g_j \in \bZ[\overline{\xi_1}, \overline{\xi_2}, \overline{\xi_1'}, \overline{\xi_2'}, \overline{\tau}, \overline{\tau'}]$ in $z\in \Stilde$. 
	
For each $j$, define holomorphic function 
	$$\hat{g}_j(w) := g_j(\bar w).$$
Then 
\begin{equation} \label{holomorphic function}
	F(z,w) := \sum_{j=1}^N f_j(z) \hat{g}_j(w)
\end{equation}
is holomorphic on $\Stilde\times \Stilde$ and vanishes identically on the real-analytic subvariety $\{w = \bar z\}$, where it coincides with \eqref{explicit relation}.  It follows that $F$ must vanish identically on $\Stilde \times \Stilde$; see \cite[Lemma 5.2]{ACZ:Betti}.  In particular, if we fix any $w_0 \in \Stilde$, we have $F(z, w_0) \equiv 0$ on $\Stilde$, and we obtain a polynomial relation in the holomorphic functions $\xi_1, \xi_2, \xi_1', \xi_2', \tau, \tau'$ that holds on all of $\Stilde$.

\medskip
\noindent{\bf Step 2:  Algebraic Independence.}
Suppose that $P_1, \ldots, P_m \in E(k)$ define a basis for $E(k)\otimes \R$, so that 
	$$X_i = \sum_{j = 1}^m a_{i, j} P_j$$
for $a_{i,j} \in \R$, $i = 1,2$.  From Proposition \ref{foliation}, we know that we can choose $\xi_i$ to satisfy 
	$$\xi_i = \sum_j a_{i,j} \xi_{P_j}$$
for choices of lifts $\xi_{P_j}$ of each point $P_j$.  From Step 1, for each $w_0 \in \Stilde$, the function \eqref{holomorphic function} satisfies $F(\cdot, w_0) \equiv 0$ on $\Stilde$, giving a polynomial relation on the holomorphic functions 
	$$\xi_{P_1}, \ldots, \xi_{P_m}, \xi_{P_1}', \ldots, \xi_{P_m}', \tau, \tau'$$
with real coefficients. 
But the functions $\xi_{P_j}$ come from the linearly independent algebraic points $P_j \in E(k)$ in the non-isotrivial $E$ and so satisfy the hypothesis of Th\'eor\`eme 5 in \cite{Bertrand}.  As a consequence of \cite[Th\'eor\`eme 5]{Bertrand}, a nontrivial polynomial relation $F(\cdot, w_0) \equiv 0$ on the functions $\xi_{P_j}$ and their derivatives $\xi_{P_j}'$ (with coefficients in the field $\C(\tau, \tau')$) implies that the points $P_j$ must themselves satisfy a nontrivial linear relation.  But this would contradict our assumption that the $P_i$ form a basis for $E(k)\otimes \R$, so we conclude that the polynomial relation must have been trivial. In other words, for any choice of $w_0$, the coefficients of $F(z,w_0)$ -- as polynomials in $\xi_{P_1}, \ldots, \xi_{P_m}, \xi_{P_1}', \ldots, \xi_{P_m}'$ -- must vanish.  
	
Examining the relation \eqref{explicit relation}, we can determine these coefficients explicitly. 
 The ``constant" term, having no dependence on the $\xi_{P_j}$ or $\xi_{P_j}'$, gives 
	$$C_1(w_0)\, \tau' + C_2(w_0)\, \tau \tau' = 0$$
as a function of $z \in \Stilde$, with coefficients $C_1, C_2$ that are antiholomorphic functions of $w_0$ on $\Stilde$.  For fixed $w_0$, if $C_1(w_0)$ or $C_2(w_0)$ is nonzero, this would imply that $\tau$ is constant, which is absurd because the elliptic surface $\mcE\to B$ is non-isotrivial.  This implies that $C_2(w_0) = 0$ for all $w_0$.  But, again looking at the formula from \eqref{explicit relation}, we have 
	$$C_2(w_0) = \overline{\xi_1'}(w_0) \overline{\xi_1}(w_0) - \overline{\xi_2'}(w_0) \overline{\xi_2}(w_0) = 0$$
for all $w_0$.  Taking complex conjugates, we have 
	$$0 \equiv \xi_1' \xi_1 - \xi_2' \xi_2 = \sum_{j, \ell=1}^m( a_{1,j}a_{1,\ell} - a_{2,j}a_{2,\ell}) \xi_{P_j}' \xi_{P_\ell}. $$
In other words, we find another relation on the holomorphic functions $\xi_{P_j}'$ and $\xi_{P_\ell}$ which must therefore be trivial \cite[Th\'eor\`eme 5]{Bertrand}.  We conclude that either
	$$a_{1, j} =  a_{2, j}$$ 
for all $j$ or that 
	$$a_{1, j} = -  a_{2, j}$$ 
for all $j$.  In other words,
	$$X_1 = \pm X_2.$$
This completes the proof of Theorem \ref{measures}.

\bigskip
\section{Proofs of the main theorems}
\label{final proofs}
	
In this section, we prove our main theorems.
	
\subsection{Proof of Theorem \ref{Lambda}}
Recall that the N\'eron-Tate height $\hat{h}_E$ on $E(k)$ extends to a positive definite quadratic form on $E(k)\otimes \R$ because $E$ is non-isotrivial.  It follows (by Cauchy-Schwarz) that the N\'eron-Tate regulator 
	$$R_E(X,Y) := \hat{h}_E(X)\hat{h}_E(Y) - \<X,Y\>_E^2 \;\geq \; 0$$ 
extends to a biquadratic form on $E(k)\otimes\R$ satisfying $R_E(X,Y) = 0$ if and only if $X$ and $Y$ are linearly dependent over $\R$.  As 
	$$F(X,Y) := \Dbar_X \cdot \Dbar_Y$$
is also biquadratic on $E(k) \otimes\R$ (see Proposition \ref{propertiesofpairing}) and satisfies $F(X,X)=0$ for all $X \in E(k)\otimes \R$, the upper bound on $\Dbar_X \cdot \Dbar_Y$ in Theorem \ref{Lambda} follows.

From Theorem \ref{TFAE}, we know that Theorem \ref{Lambda} holds for $\mcE\to B$ if and only if $\Dbar_X \cdot \Dbar_Y \not=0$ for all pairs of linearly independent $X, Y \in E(k)\otimes \R$.  So assume we have nonzero elements $X, Y \in E(k) \otimes \R$ satisfying $\Dbar_X \cdot \Dbar_Y =0$.  By scaling $X$ and $Y$, we may assume that $\hat{h}_E(X) = \hat{h}_E(Y) = 1$.  We proved in Theorem \ref{real point divisor theorem} that $\Dbar_X$ and $\Dbar_Y$ are normalized, semipositive, continuous adelic metrizations on $\R$-divisors on $B$, each on divisors of degree 1.  Theorem \ref{isomorphic} then implies that $\Dbar_X$ and $\Dbar_Y$ are isomorphic, so the curvature forms for $\Dbar_X$ and for $\Dbar_Y$ on $B_v^{\an}$ must coincide at all places $v$ of the number field $K$.  Fixing a single archimedean place, we deduce from Theorem \ref{measures} that $X = \pm Y$.  This completes the proof.  
	
\subsection{Proof of Theorem \ref{scheme}}
Suppose that $C$ is an algebraic curve in $\mcE^m$ that dominates the base curve $B$.  Passing to a finite branched cover $B' \to B$, we may view $C$ as a section $C'$ of the $m$-th fibered power of the pull-back elliptic surface $\mcE' \to B'$.  As Theorem \ref{Lambda} holds for $\mcE'\to B'$, we apply Theorem \ref{TFAE} to conclude that the intersection of $C'$ with the tube $T((\mcE')^{m, \{2\}}, \epsilon)$ is contained in a finite union of flat subgroup schemes of positive dimension, for all sufficiently small $\epsilon>0$.  Projecting back to $\mcE^m\to B$, we can make the same conclusion about the intersection of $C$ with $T(\mcE^{m, \{2\}}, \epsilon)$.  This completes the proof.

	
\bigskip
\appendix
	
\section{Arithmetic equidistribution for $\R$-divisors}
In this Appendix, we show that an equidistribution law holds on projective varieties defined over a number field, for adelic semipositive metrizations $\Dbar$ associated to an ample $\R$-divisor.  Formal definitions, extending those we provided for curves in Section \ref{R-divisors}, appear in \cite[Chapters 2 and 4]{Moriwaki:Memoir}. (Note that our definition of $D$-Green function differs from the one in \cite{Moriwaki:Memoir} by a factor of $2$.) Theorem \ref{two bundles} and Corollary \ref{equidistribution} extend the equidistribution theorems of Chambert-Loir, Thuillier, and Yuan \cite{ChambertLoir:equidistribution, Thuillier:these, Yuan:equidistribution} for adelically metrized line bundles to $\R$-divisors.  Our proofs follow a known strategy for equidistribution; we mimic the presentation of Chambert-Loir and Thuillier in \cite{ChambertLoir:Thuillier}, while they appeal to results of Yuan \cite{Yuan:equidistribution} and Zhang \cite{Zhang:adelic}, building on the ideas that originally appeared in \cite{Szpiro:Ullmo:Zhang}.  See also  \cite{Yuan:survey}.  We provide the details for completeness.  The key ingredient for passing from $\bQ$-divisors to $\R$-divisors is the continuity of the arithmetic volume function on the space of metrized of $\R$-divisors, proved by Moriwaki \cite[Theorem 5.3.1]{Moriwaki:Memoir}.

\begin{theorem}  \label{two bundles}
Let $X$ be a normal and geometrically integral projective variety of dimension $d\geq 1$ over a number field $K$.  Fix an ample $\bR$-divisor $D$ on $X$, equipped with a continuous, relatively nef, adelic metrization $\Dbar$ over $K$, satisfying $\widehat{\deg}(\Dbar^{d+1}) = 0$.  Let $\Mbar$ be an integrable adelic metrization on an $\R$-divisor $M$ over $K$.  For any generic sequence $x_n\in X(\Kbar)$ with $h_{\Dbar}(x_n) \to 0$, we have
	$$h_{\Mbar}(x_n) \to \frac{\widehat{\deg}( \Dbar^d \, \Mbar)}{\mathrm{vol}(D)}.$$
\end{theorem}
	
A sequence $\{x_n\} \subset X(\Kbar)$ is generic if every subsequence is Zariski dense.  The arithmetic notions of relatively nef and integrable are defined in \cite[\S4.4]{Moriwaki:Memoir}, and the multilinear, symmetric intersection form $\widehat{\deg}(\Dbar_1\cdots \Dbar_{d+1})$ is defined in \cite[\S4.5]{Moriwaki:Memoir}.  The intersection coincides with the arithmetic intersection number denoted by $c_1(\Lbar_1)\cdots c_1(\Lbar_{d+1})$ in \cite{Zhang:adelic} when $\Dbar_i$ is the metrized divisor associated to an adelically metrized line bundle $\Lbar_i$; see Remark \ref{intersectionvsdeg}.

For curves $X$, the hypothesis on $\Dbar$ in Theorem \ref{two bundles} simplifies in the language of Section \ref{R-divisors} to a continuous, semipositive, and normalized metrization.  We have $\widehat{\deg}( \Dbar^d \, \Mbar) = \Dbar \cdot \Mbar$ as defined in \eqref{pairing}.  
	
\begin{cor}  \label{equidistribution}
Let $X$ be a normal and geometrically integral projective variety of dimension $d\geq 1$ over a number field $K$.  Fix an ample $\bR$-divisor $D$ on $X$, equipped with a continuous, relatively nef adelic metrization $\Dbar$ over $K$, satisfying $\widehat{\deg}(\Dbar^{d+1}) = 0$.  For each place $v$ of $K$ and for any generic sequence $x_n\in X(\Kbar)$ with $h_{\Dbar}(x_n) \to 0$, the discrete probability measures
	$$\mu_n = \frac{1}{|\Gal(\Kbar/K)\cdot x_n|} \sum_{y \in \Gal(\Kbar/K)\cdot x_n}  \delta_y$$
converge weakly in $X^{\an}_v$ to the probability measure 
	$$\mu_{\Dbar,v} = \frac{1}{\vol(D)} \, c_1(\Dbar)_v^d$$ 
\end{cor}
	
\noindent 
Here, $X_v^{\an}$ denotes the Berkovich analytification of the variety $X$ over the complete and algebraically closed field $\C_v$.  The measure $c_1(\Dbar_1)_v\cdots c_1(\Dbar_d)_v$ is defined in \cite{ChambertLoir:equidistribution} for integrable, adelically metrized line bundles on $X$, and the definition extends to $\R$-divisors by multilinearity.  For curves $X$, we have $d=1$ and $c_1(\Dbar)_v = \omega_{\Dbar,v}$ as defined in \S\ref{metrizations}.
	
\subsection{Essential minima}  \label{d minima}
Let $X$ be a normal and geometrically integral projective variety of dimension $d\geq 1$ over a number field $K$.  For any $\R$-divisor $D$ on $X$ defined over $K$, we set 
	$$H^0(X, D) = \{\phi \in K(X):  (\phi) + D \geq 0\} \cup \{0\}.$$
For ample $D \in \mathrm{Div}_{\bZ}(X)$, the volume of $D$ is  
	$$\vol D = \lim_{k\to\infty} \frac{d!}{k^d} \dim H^0(X, kD).$$
For a $\bQ$-divisor $D$, the volume can be defined by taking the limit along sequences where $kD \in \mathrm{Div}_{\bZ}(B)$.  The volume extends continuously to $\R$-divisors; see, for example, \cite[Theorem 2.2.44]{Lazarsfeld:Positivity:I}. 

As in \S\ref{minima} and following \cite{Zhang:adelic}, the essential minimum of the height $h_{\Dbar}$ is defined as 
	$$	e_1(\Dbar) := \sup_{Y} \inf_{x \in (X\setminus Y) (\Kbar)} h_{\Dbar}(x),$$
over all Zariski closed proper subsets $Y$ in $X$ of codimension 1, and we put 
	$$e_{d+1}(\Dbar) := \inf_{x\in X(\Kbar)} h_{\Dbar}(x).$$

\begin{theorem} \cite[Theorem 1.10]{Zhang:adelic} \label{Zhang inequality d} 
For any adelic, semipositive metrization $\Dbar$ of an ample $\R$-divisor $D$ on $X$, we have
	$$e_1(\Dbar) \geq \; \frac{\widehat{\deg}(\Dbar^{d+1})}{(d+1) \vol D} \;  \geq \frac{1}{d+1} \left( e_1(\Dbar) + d\, e_{d+1}(\Dbar) \right)$$
\end{theorem}
	
\begin{proof}
Zhang proved the result for ample line bundles equipped with adelic, semipositive metrics \cite[Theorem 1.10]{Zhang:adelic}.  It holds also for metrizations of $\R$-divisors because the height function associated to an $\R$-divisor is a uniform limit of heights associated to $\bQ$-divisors, and the intersection number is multilinear and the volume $\vol(D)$ is continuous.
\end{proof}

\subsection{Arithmetic volume}  Let $X$ be a normal and geometrically integral projective variety of dimension $d\geq 1$ over a number field $K$.  	The {\bf arithmetic volume} of an adelically metrized $\R$-divisor $\Dbar$ is defined as follows.  We first fix a family of norms on $H^0(X,D)$ by 
$$\|\phi\|_{\mathrm{sup},v} = \displaystyle\sup_{x \in X_v^{\an}\setminus\supp D} |\phi(x)|_v e^{-g_v(x)},$$
for each place $v$ of $K$.  
Set
$$\chi(\Dbar) = - \log \frac{ \mu((H^0(X,D)\otimes {\bf A}_K)/H^0(X,D))}{ \mu(\prod_v U_v)}$$
where ${\bf A}_K$ is the ring of adeles, $\mu$ is a Haar measure on $H^0(X,D)\otimes {\bf A}_K$, and $U_v$ is the unit ball in $H^0(X,D)\otimes \C_v$ in the induced norm.  Then
	$$\widehat{\mathrm{vol}}_\chi(\Dbar) := \limsup_{k\to\infty}  \frac{(d+1)!}{k^{d+1}} \, \chi(k\Dbar).$$
	
In  \cite[Theorem 5.2.1]{Moriwaki:Memoir}, Moriwaki proves that $\widehat{\mathrm{vol}}_\chi$ defines a continuous function on a space of continuous, adelic metrizations on $\R$-divisors.  As a consequence, he shows that for relatively nef metrizations, we have $\widehat{\mathrm{vol}}_\chi(\Dbar) = \widehat{\deg}(\Dbar^{d+1})$ \cite[Theorem 5.3.2]{Moriwaki:Memoir}.  
Therefore, Zhang's inequality (Theorem \ref{Zhang inequality d}) implies that 
\begin{equation}  \label{Zhang positive d}
	e_1(\Dbar) \geq \widehat{\mathrm{vol}}_\chi(\Dbar)/( (d+1) \vol D).
\end{equation} 
for all continuous, semipositive, adelic metrizations of $\R$-divisors on $B$.

\begin{remark}  This volume function $\widehat{\mathrm{vol}}_\chi$ is defined differently than the one studied by Moriwaki in \cite{Moriwaki:Memoir}, but they coincide.  See, for example, Appendix C.2 of \cite{Bombieri:Gubler} and the discussion on page 615, for the comparison of an adelic volume to a Euclidean volume.  
\end{remark}
	
\begin{proposition} \label{strong Zhang}
For all integrable adelic metrizations on an ample $\R$-divisor $D$, we have 
	$$e_1(\Dbar) \geq \frac{\widehat{\mathrm{vol}}_\chi(\Dbar)}{(d+1) \, \vol D}$$  
\end{proposition} 
	
\begin{proof}
From \eqref{Zhang positive d}, the inequality holds for relatively nef $\Dbar$. For integrable metrics, we write $\Dbar = \Dbar_1- \Dbar_2$ for relatively nef $\Dbar_i$ and approximate each $\Dbar_i$ with relatively nef adelic metrics on $\bQ$-divisors $\Dbar_{i,n}$ as $n\to \infty$.  Because $D$ is ample, we can assume that $D_{1,n} - D_{2,n}$ is ample for all $n$.  In that setting, we apply \cite[Lemme 5.1]{ChambertLoir:Thuillier}. The result then follows by uniform convergence of the resulting height functions, so that $e_1$ is continuous, and by continuity of the volume function $\widehat{\mathrm{vol}}_\chi$  \cite[Theorem 5.2.1]{Moriwaki:Memoir} and of the classical volume.
\end{proof}

\subsection{Proof of Equidistribution}  
	
\begin{proof}[Proof of Theorem \ref{two bundles}]
Fix an ample $\R$-divisor $D$, equipped with an adelic, relatively nef metrization $\Dbar$ for which $\widehat{\deg}(\Dbar^{d+1}) = 0$.  Let $x_n \in X(\Kbar)$ be a generic sequence with $h_{\Dbar}(x_n) \to 0$.  
		
Assume first that $\Mbar$ is an adelic, {\em arithmetically nef} metrization on an ample $\R$-divisor $M$, meaning that $\Mbar$ is relatively nef and the height $h_{\Mbar}$ is non-negative at all points of $X(\Kbar)$; see \cite[\S4.4]{Moriwaki:Memoir}.  For each positive integer $m$, by Zhang's inequality (Theorem \ref{Zhang inequality d}) applied to $(m\Dbar) + \Mbar$, we have 
	$$\liminf_{n\to\infty} \left( m h_{\Dbar}(x_n) + h_{\Mbar}(x_n)\right) \geq \frac{\deghat((m\Dbar + \Mbar)^{d+1})}{(d+1)\vol(m\, D +  M)} = \frac{ (d+1)m^d \,\deghat(\Dbar^d\,\Mbar) + O(m^{d-1}) }{(d+1)\vol(m\, D +  M)}$$
from the multilinearity of the intersection number and because $\deghat(\Dbar^{d+1}) = 0$.
As the sequence $x_n$ is small for $\Dbar$, this gives 
	$$\liminf_{n\to\infty} h_{\Mbar}(x_n) \geq \frac{ (d+1)m^d\, \deghat(\Dbar^d\,\Mbar) + O(m^{d-1}) }{(d+1)\vol(m\, D +  M)}$$
for all $m$.  Letting $m$ go to $\infty$, we obtain
\begin{align}\label{Zhanglower}
	\liminf_{n\to\infty}  h_{\Mbar}(x_n) \geq  \frac{\deghat(\Dbar^d \, \Mbar)}{\vol D}.
\end{align}
		
For the reverse inequality, we choose $m$ large enough so that $mD - M$ is ample.  We can therefore apply Proposition \ref{strong Zhang} to obtain
	$$\liminf_{n\to \infty} ( m h_{\Dbar}(x_n) - h_{\Mbar}(x_n))  \geq  \frac{ \widehat{\mathrm{vol}}_\chi \left(m\Dbar- \Mbar\right)}{(d+1)\, \vol(mD-M)}$$
so that 
\begin{equation} \label{limsup}
	- \limsup_{n\to \infty} h_{\Mbar}(x_n)  \geq  \frac{ \widehat{\mathrm{vol}}_\chi \left(m\Dbar- \Mbar\right)}{(d+1)\, \vol(mD-M)}
\end{equation}
Fix a place $v_0$ of the number field $K$ and $c\in \R$, and let $\Dbar_c$ denote $\Dbar + (0, \{g_v\})$ where $g_{v_0}(x) \equiv c/r_v$ and $g_v(x)\equiv 0$ for all $v \not= v_0$; recall that $r_v$ was defined in \eqref{rv}.  Choosing $c$ large enough, we can assume that $\Dbar_c$ is arithmetically nef.  It follows that 
	$$\volhat_\chi(m\Dbar_c-\Mbar) \geq m^{d+1} \deghat(\Dbar_c^{d+1}) - (d+1) m^d \deghat(\Dbar_c^d \,\Mbar),$$
combining \cite[Theorem 2.2]{Yuan:equidistribution} with the continuity of $\volhat_\chi$ \cite[Theorem 5.2.1]{Moriwaki:Memoir}; see also \cite[Lemme 5.2]{ChambertLoir:Thuillier}.  
		
But note that 
	$$\volhat_\chi(m\Dbar_c-\Mbar) = \volhat_\chi (m\Dbar - \Mbar) + (d+1)\, c\, m  \vol(mD-M)$$
from the definition of $\volhat_\chi$.  Consequently, 
\begin{eqnarray*}
	\widehat{\mathrm{vol}}_\chi \left(m\Dbar- \Mbar\right) &\geq& m^{d+1} \deghat(\Dbar_c^{d+1}) - (d+1) m^d \deghat(\Dbar_c^d \,\Mbar) -  (d+1)c \, m  \vol(mD-M) \\
	&=&  m^{d+1}\left(  \deghat(\Dbar^{d+1}) + c(d+1)\vol(D)\right) \\
	&&  \quad - \; (d+1)m^d\left( \deghat(\Dbar^d \, \Mbar) + d\,c \, c_1(D)^{d-1}c_1(M) \right) \\
	&& \quad - \; (d+1)c \, m \vol(mD-M) \\
	&=& -(d+1)m^d \,\deghat(\Dbar^d\, \Mbar) + O(m^{d-1})
\end{eqnarray*}
with the last equality because $\deghat(\Dbar^{d+1}) = 0$.  Compare \cite[Proposition 5.3]{ChambertLoir:Thuillier}.  
		
Therefore, \eqref{limsup} gives
	$$ \limsup_{n\to \infty} h_{\Mbar}(x_n)  \leq - \frac{\widehat{\mathrm{vol}}_\chi (m\Dbar- \Mbar)}{(d+1) \vol(mD-M)} \leq \frac{(d+1)m^d \,\deghat(\Dbar^d\, \Mbar) + O(m^{d-1})}{(d+1) \vol(mD-M)}$$
for all sufficiently large $m$.  Letting $m\to \infty$, we obtain the desired upper bound of 
\begin{align}\label{Yuanupper}
	\limsup_{n\to \infty} h_{\Mbar}(x_n)  \leq \frac{\deghat(\Dbar^d \, \Mbar)}{\vol D}.
\end{align}
Putting the two inequalities \eqref{Zhanglower} and \eqref{Yuanupper} together, we have
	$$\lim_{n\to\infty} h_{\Mbar}(x_n) = \frac{\deghat(\Dbar^d \, \Mbar)}{\vol D}.$$
		
Now suppose that $\Mbar$ is integrable.  By definition, we can write $\Mbar = \Mbar_1 - \Mbar_2$ for relatively nef $\Mbar_i$ on ample divisors $M_i$.  By adding and subtracting the trivial divisor with constant metric, we can assume that each $\Mbar_i$ is arithmetically nef, and we apply the result above to each $\Mbar_i$.  We have $h_{\Mbar} = h_{\Mbar_1} - h_{\Mbar_2}$ and $\deghat(\Dbar^d \, \Mbar) = \deghat(\Dbar^d \, \Mbar_1) - \deghat(\Dbar^d \, \Mbar_2)$.  The theorem is a consequence of this linearity.  
\end{proof} 
	
\medskip
	
\begin{proof}[Proof of Corollary \ref{equidistribution}]
	Fix a place $v\in M_K$, and let $\phi$ be a smooth real-valued function on $X_v^{\an}$.  By density as in \cite[Theorem 3.3.3]{Moriwaki:Memoir} it is enough to consider these functions.  We denote  by $\overline{O}_\phi$ the trivial divisor on $X$ equipped with the metrization given by $g_v = \phi$ and $g_w = 0$ for all $w\not= v$ in $M_K$.  This metrization is integrable.

Let $\mu_n$ denote the probability measure in $X^{\an}_v$ supported uniformly on the Galois conjugates of $x_n$.  Note that
	$$h_{\overline{O}_\phi}(x_n)= r_v \int_{X^{\an}_v}\phi \, d\mu_{n}$$
by the definition of the height function, where $r_v = [K_v:\bQ_v]/[K:\bQ]$.  We have 
	$$\deghat(\Dbar^d \, \overline{O}_\phi)= r_v\int_{X^{\an}_v}\phi \; c_1(\Dbar)_v^d.$$
Applying Theorem \ref{two bundles} to $\Mbar=\overline{O}_\phi$, we get that 
	$$\lim_{n\to\infty}h_{\overline{O}_\phi}(x_n)= \frac{r_v}{\vol D} \int_{X^{\an}_v}\phi \; c_1(\Dbar)_v^d = r_v \int_{X^{\an}_v}\phi \,d\mu_{\Dbar,v},$$
demonstrating weak convergence of $\mu_n$ to $\mu_{\Dbar,v}$ in $B^{\an}_v$.
\end{proof}


\bigskip \bigskip

\def\cprime{$'$}

\bigskip\bigskip

\end{document}